\documentclass{amsart}

\usepackage{amssymb}
\usepackage{amsmath}
\usepackage{amsthm}
\usepackage[all]{xy}
\usepackage{esint}

\newtheorem{thm}{Theorem}[section]

\newtheorem{cor}[thm]{Corollary}
\newtheorem{lem}[thm]{Lemma}
\newtheorem{prop}[thm]{Proposition}
\newtheorem*{quest*}{Question}
\theoremstyle{definition}
\newtheorem{defn}[thm]{Definition}
\newtheorem{ex}[thm]{Example}

\theoremstyle{remark}
\newtheorem{rem}[thm]{Remark}
\newcommand{\ZZ}{\mathbb{Z}}
\newcommand{\RR}{\mathbb{R}}

\newcommand{\bt}{\bullet}
\newcommand{\G}{\Gamma}
\renewcommand{\L}{\Lambda}

\newcommand{\mg}{\mathfrak{g}}
\newcommand{\U}{\mathcal{U}}

\newcommand{\F}{\mathcal{F}}
\newcommand{\E}{\mathcal{E}}
\renewcommand{\P}{\mathcal{P}}
\renewcommand{\H}{\mathcal{H}}
\newcommand{\Om}{\mathcal{O}}

\newcommand{\I}{I}
\newcommand{\K}{K}

\renewcommand{\k}{\mathfrak{k}}

\DeclareMathOperator{\Lie}{Lie}
\DeclareMathOperator{\Tr}{Tr}
\DeclareMathOperator{\supp}{supp}
\DeclareMathOperator{\Dom}{Dom}
\DeclareMathOperator{\image}{Im}
\DeclareMathOperator{\SL}{SL}
\DeclareMathOperator{\GL}{GL}
\DeclareMathOperator{\SO}{SO}
\DeclareMathOperator{\so}{\mathfrak{so}}

\DeclareMathOperator{\Aut}{Aut}
\DeclareMathOperator{\Hol}{Hol}
\DeclareMathOperator{\hol}{hol}
\DeclareMathOperator{\dev}{dev}
\DeclareMathOperator{\inv}{inv}
\DeclareMathOperator{\oO}{\overline{O}}
\DeclareMathOperator{\rank}{rank}

\DeclareMathOperator{\orr}{o}

\newtheorem*{ack}{Acknowledgment}

\numberwithin{equation}{section}

\begin{document}

\title{Haefliger cohomology of complete Riemannian foliations}
\author{Hiraku Nozawa}
\email{hnozawa@fc.ritsumei.ac.jp}
\address{Department of Mathematical Sciences,
Colleges of Science and Engineering, Ritsumeikan University,
1-1-1 Nojihigashi, Kusatsu, Shiga, 525-8577, Japan}

\thanks{The author is partly supported by the EPDI/JSPS/IH\'{E}S Fellowship, the Spanish MICINN grant MTM2011-25656, MTM2014-56950-P, JSPS KAKENHI Grant Number 17K14195 and Sabbatical Grant for young researchers of Ritsumeikan University.}

\begin{abstract}
Haefliger cohomology characterizes taut foliated manifolds by Haefliger's theorem. We show that Haefliger cohomology characterizes strongly tense foliated manifolds, namely, foliated manifolds which admit a Riemannian metric such that the mean curvature form of the leaves is closed and basic. We show that Haefliger cohomology is dual to invariant cohomology for complete Riemannian foliations. As an application of these results, we prove that any complete Riemannian foliation is strongly tense, which is a generalization of Dom\'{i}nguez's tenseness theorem for Riemannian foliations on closed manifolds. 
\end{abstract}

\maketitle

\tableofcontents

\addtocontents{toc}{\protect\setcounter{tocdepth}{1}}

\section{Introduction}
\subsection{Background and outline}

A foliated manifold $(M,\mathcal{F})$ is called {\em taut} if $M$ admits a metric $g$ such that every leaf of $\mathcal{F}$ is a minimal submanifold of $(M,g)$. Haefliger cohomology of foliated manifolds was introduced by Haefliger~\cite{Haefliger1980} to study tautness of compact foliated manifolds. Based on the work of Sullivan~\cite{Sullivan1979} and Rummler~\cite{Rummler1979}, Haefliger characterized tautness for a compact foliated manifold $(M,\F)$ by the existence of certain nonnegative $0$-cocycle of the Haefliger cohomology. As a consequence, he showed that tautness of compact foliated manifold $(M,\F)$ is a transverse property, which is determined by the holonomy pseudogroup. 

Remarkably, tautness of Riemannian foliations is of topological nature being characterized in terms of basic cohomology as conjectured by Carri\'{e}re in his Ph.D.\ thesis and proven by Masa~\cite{Masa1992}. Later, based on Masa's technique and the \'{A}lvarez class defined in~\cite{AlvarezLopez1992}, Dom\'{i}nguez~\cite{Dominguez1998} generalized Masa's theorem by proving that any Riemannian foliation on a compact manifold is tense. Here, recall that a foliated manifold $(M,\mathcal{F})$ is called {\em tense} if $M$ admits a metric such that the mean curvature form of $\mathcal{F}$ is basic. The tenseness naturally arises to prove the duality of cohomologies of Riemannian foliations on closed manifolds~\cite{KamberTondeur1983b,Dominguez1998}.

 In this article, we study strong tenseness of foliated manifolds, a notion stronger than tenseness which is useful to study Riemannian foliations, to characterize it with Haefliger cohomology (Theorem~\ref{thm:Hcharint}). Combining this result with the duality between Haefliger cohomology and invariant cohomology of complete Riemannian foliations (Theorems \ref{thm:masa} and \ref{thm:dfol}), we generalize Dom\'{i}nguez's tenseness theorem to complete Riemannian foliations (Corollary~\ref{cor:st}). These results are illustrated with examples of Lie foliations in Section \ref{sec:exam}. Our computation of Haefliger cohomology is applied to the localization of some Chern-Simons type invariants of Riemannian foliations in \cite{GoertschesNozawaToben2017}. Our original motivations are recent studies \cite{RoyoPrietoSaralegiWolak2008,Masa2009,RoyoPrietoSaralegiWolak2009,NozawaRoyoPrieto2014} on Riemannian foliations on open manifolds as well as some significant differences on the tautness between Riemannian foliations on closed manifolds and those on open manifolds (see Section \ref{sec:9}).

\subsection{Strong tenseness and Haefliger cohomology}\label{sec:I2}

Let us recall strong tenseness of foliated manifolds introduced in \cite{NozawaRoyoPrieto2014}: A foliated manifold $(M,\F)$ is called {\em strongly tense} if $M$ admits a Riemannian metric so that the mean curvature form $\kappa$ of the leaves of $\F$ is basic and closed. It is easy to see that  $(M,\F)$ is strongly tense and $\kappa$ is exact, then $(M,\F)$ is taut.

We will give a Haefliger type characterization of strong tenseness of foliated manifolds. Let us give the definition of Haefliger cohomology of a foliated manifold $(M,\F)$ with values in a flat vector bundle $\pi_{\E} : (\E,\nabla) \to M$. Take a total transversal $T$ of $(M,\F)$ and let $P_{T}(\F)$ be the set of all leaf paths on $(M,\F)$ which connects two points on $T$:
\[
P_{T}(\F) = \{ \,  \gamma : [0,1] \to M \mid \gamma (0), \gamma (1) \in T, \exists L \in \F, \gamma([0,1]) \subset L \, \}.
\] 
For any $\gamma \in P_{T}(\F)$, we have an open neighborhood $U_{\epsilon}$ of $\gamma (\epsilon)$ in $T$ $(\epsilon = 0,1)$ and an isomorphism $h_{\gamma} : (\E,\nabla)|_{U_{0}} \to (\E,\nabla)|_{U_{1}}$ induced by the holonomy map of $\F$ along $\gamma$. We call the pseudogroup $\Hol_{T}(\F,\E)$ on $\E|_{T}$ generated by $\{h_{\gamma}\}_{\gamma \in P_{T}(\F)}$ the {\em holonomy pseudogroup on} $T$ {\em of the pair} $(\F,\E)$. Here $\Hol_{T}(\F,\E)$ acts on $\Omega^{\bt}_{c}(T;\E|_{T})$ by pull-back. Let $\Omega^{\bt}_{c}(\Tr \F;\E)$ denote the complex $\Omega^{\bt}_{c}(T;\E|_{T}) / \Lambda^{\bt}_{c}$, where 
\[
\Lambda^{\bt}_{c} = \{\, {\scriptstyle \sum_{1 \leq i \leq k}} h_{i}^{*} \sigma_{i} - \sigma_{i} \mid h_{i} \in \Hol_{T}(\F,\E), \sigma_{i} \in \Omega^{\bt}_{c}(T;\E|_{T}), \pi^{-1}_{\E}(\operatorname{supp} \sigma_{i}) \subset \image h_{i} \, \}
\]
with the covariant derivative of $(\E,\nabla)|_{T}$ endowed with the quotient topology induced from the standard LF-topology on $\Omega^{\bt}_{c}(T;\E|_{T})$ (see~\cite[Section~9]{deRham1984}). Here $\Omega^{\bt}_{c}(\Tr \F;\E)$ does not depend on the choice of $T$. The cohomology $H^{\bt}_{c}(\Tr \F;\E)$ of this complex is called the {\em Haefliger cohomology} (or transverse cohomology) of $(M,\F)$ with values in $(\E,\nabla)$ (see Section~\ref{sec:Hcoh}). 

We will use the following terminologies in the case where $\rank \E = 1$. We say that $\xi \in \Omega^{0}_{c}(\Tr \F;\E)$ is {\em nonnegative} if a representative $\zeta$ of $\xi$ in $\Omega^{0}_{c}(T;\E|_{T})$ is nonnegative with respect to a topological trivialization of $\E$. For a nonnegative element $\xi \in \Omega^{0}_{c}(\Tr \F;\E)$, let $\supp_{+}^{\F} \xi \subset M$ be the union of leaves which intersect $\{ x \in T \mid \zeta(x) \neq 0\}$ for any nonnegative representative $\zeta \in \Omega^{0}_{c}(T;\E|_{T})$ of $\xi$. Consider a cohomology class $\omega_{\nabla} \in H^{1}(M;\RR)$ given by $\log |h(\gamma)| = \int_{\gamma}\omega_{\nabla}$ for $\forall \gamma \in \pi_{1}M$, where $h : \pi_{1}M \to \Aut(\RR) \cong \RR^{\times}$ is the holonomy homomorphism of $(\E,\nabla)$. We say that $\E$ has {\em basic holonomy} if the cohomology class $\omega_{\nabla}$ belongs to the basic cohomology $H^{1}(M,\F)$.

The following is our Haefliger type characterization of strong tenseness. 

\begin{thm}\label{thm:Hcharint}
Let $\F$ be an oriented foliation on a closed manifold $M$. The following are equivalent:
\begin{enumerate}
\item $(M,\F)$ is strongly tense.
\item There exist a topologically trivial flat real line bundle $(\E,\nabla)$ with basic holonomy over $(M,\F)$ and a nonnegative $0$-cocycle $\xi$ in $\Omega^{0}_{c}(\Tr \F;\E)$ such that $\supp_{+}^{\F}\xi = M$.
\end{enumerate}
\end{thm}

Theorem \ref{thm:Hcharint} is proven with a Rummler-Sullivan type characterization (Proposition~\ref{prop:RS}), which shows that strong tenseness is characterized with the existence of a certain characteristic form twisted by a flat real line bundle. Since the condition (ii) is invariant under the equivalence of holonomy pseudogroups of foliated manifolds, we obtain the following.

\begin{cor}\label{cor:eq}
Strong tenseness of closed foliated manifolds is a transverse property, i.e., it is determined by the equivalence class of the holonomy pseudogroup.
\end{cor}

We will show that one direction of Theorem \ref{thm:Hcharint} extends to open foliated manifolds whereas the other direction fails (see Section \ref{sec:9}). For this purpose, we will introduce the leafwise compactly supported version $H^{\bt}_{lc}(\Tr \F;\E)$ of Haefliger cohomology (see Section \ref{sec:charst}).

\begin{thm}\label{thm:domi}
Let $(M,\F)$ be a possibly open foliated manifold. If there exist a topologically trivial flat real line bundle $(\E,\nabla)$ with basic holonomy over $(M,\F)$ and a nonnegative $0$-cocycle $\xi$ in $\Omega^{0}_{lc}(\Tr \F;\E)$ such that $\supp_{+}^{\F} \xi = M$, then $(M,\F)$ is strongly tense.
\end{thm}

Here we briefly review the motivation to consider strong tenseness of foliated manifolds. As seen in the work of Kamber-Tondeur~\cite{KamberTondeur1983} and Dom\'{i}nguez~\cite{Dominguez1998}, tenseness of Riemannian foliations on closed manifolds is a cohomological property which can be regarded as a twisted version of tautness. It is crucial in their work that, on a Riemannian foliation on a closed manifold, the mean curvature form of any tense metric is closed~\cite[Eq.\ 4.4]{KamberTondeur1983}. However tenseness is not a cohomological property for general foliated manifolds. The example of a Riemannian foliation on an open manifold due to Cairns-Escobales~\cite[Example~2.4]{CairnsEscobales1997} admits a metric whose mean curvature form is basic but not closed. For such foliations, it is difficult to relate tenseness with cohomology like in the work of Haefliger~\cite{Haefliger1980}. The strong tenseness was introduced in~\cite{NozawaRoyoPrieto2014} being motivated by this fact. By Kamber-Tonduer's result \cite[Eq.~4.4]{KamberTondeur1983}, any tense Riemannian foliation on closed manifolds are strongly tense. For detailed discussion on the mean curvature and cohomology of Riemannian foliations, see \cite{AlvarezLopez1992,MarchMinOoRuh1996,Mason2000,HabibRichardson2013}.

\subsection{Haefliger cohomology of complete Riemannian foliations}\label{sec:I1}

Haefliger cohomology is of dynamical nature in general as seen in~\cite{Haefliger1980}, and it is often difficult to compute them. We will show that Haefliger cohomology of complete Riemannian foliations is isomorphic to the reduced Haefliger cohomology, which is less difficult to compute. Here the reduced cohomology $\widehat{H}^{\bt}_{c}(\Tr \F;\E)$ (resp.\ $\widehat{H}^{\bt}_{lc}(\Tr \F;\E)$) is the cohomology of the quotient of $\Omega^{\bt}_{c}(\Tr \F;\E)$ (resp.\ $\Omega^{\bt}_{lc}(\Tr \F;\E)$) by the closure of zero with respect to the LF-topology (resp. the local LF-topology). There are natural maps $H^{\bt}_{c}(\Tr \F;\E) \to \widehat{H}^{\bt}_{c}(\Tr \F;\E)$ and $H^{\bt}_{lc}(\Tr \F;\E) \to \widehat{H}^{\bt}_{lc}(\Tr \F;\E)$. For any Riemannian foliation $\F$ on a closed manifold $M$, these natural maps are isomorphisms due to Dom\'{i}nguez~\cite[Theorem~3.15]{Dominguez1998}, which is a generalization of Masa's result~\cite[Section~1]{Masa1992} in the case where $(\E,\nabla)$ is trivial and Hector's result~\cite[Theorem~1.2]{Hector1988} in the case where the structural Lie algebra of $(M,\F)$ is compact or nilpotent and $(\E,\nabla)$ is trivial. 

We need completeness of pseudogroups (see Definition~\ref{defn:comp}) to prove the isomorphism between two cohomologies, because our argument relies on Salem's structure theorem for complete Riemannian pseudogroups (Theorem \ref{thm:Molino}). We call a flat vector bundle $\E$ over a foliated manifold $(M,\F)$ {\em complete} if the holonomy pseudogroup $\Hol_{T}(\F,\E)$ on $T$ of the pair $(\F,\E)$ is complete for some total transversal $T$ of $\F$. Note that the completeness of $\E$ implies the completeness of the holonomy pseudogroup of $\F$. 

\begin{thm}\label{thm:masa}
For a Riemannian foliation $\F$ on a manifold $M$ with a complete flat vector bundle $(\E,\nabla)$, the natural maps induce isomorphisms $H^{\bt}_{c}(\Tr \F;\E) \cong \widehat{H}^{\bt}_{c}(\Tr \F;\E)$ and $H^{\bt}_{lc}(\Tr \F;\E) \cong \widehat{H}^{\bt}_{lc}(\Tr \F;\E)$.
\end{thm}
Our proof is based on a transverse version of Sarkaria's smoothing operator~\cite{Sarkaria1978} used by~\cite{Masa1992,Dominguez1998}. As a byproduct of the proof, we will show the following.
\begin{thm}\label{thm:findim}
For  a Riemannian foliation $\F$ on a manifold $M$ with a complete flat vector bundle $(\E,\nabla)$ such that the space of leaf closures $M/\overline{\F}$ is compact, we have $\dim H^{\bt}_{c}(\Tr \F;\E) < \infty$.
\end{thm}
Haefliger cohomology of foliated manifolds is a part of the $E_{2}$-terms of the associated spectral sequence  (see, for example,~\cite[Section~28]{AlvarezLopezMasa2008}). For Riemannian foliations on closed manifolds, it was proven that the whole associated spectral sequence is finite dimensional by \'{A}lvarez L\'{o}pez in the case of trivial coefficient (\cite{AlvarezLopez1989}, see also Masa \cite{Masa1992}) and by Dom\'{i}nguez \cite{Dominguez1998} in the general case with a flat vector bundle. Note that it is easy to see that, if $M/\overline{\F}$ is compact, then $H^{\bt}_{lc}(\Tr \F;\E) \cong H^{\bt}_{c}(\Tr \F;\E)$.

\subsection{Duality between Haefliger cohomology and basic cohomology}

The basic cohomology $H^{\bt}(M,\F)$ of a foliated manifold $(M,\F)$ with a flat vector bundle $(\E,\nabla)$ is the cohomology of the complex of basic forms
\[
\Omega^{\bt}(M,\F;\E)=\{ \, \sigma\in \Omega^{\bt}(M;\E)\mid \iota (X)\sigma=0, \theta(X)\sigma=0, \forall X \in C^{\infty}(T\F) \, \}
\]
with the restriction of the differential of $\Omega^{\bt}(M;\E)$, where $\theta$ denotes the Lie derivative. This cohomology is well understood and useful for Riemannian foliations, whereas it may be pathological for general foliations. 

We use the cohomology $H^{\bt}_{c}(M,\F)$ of the complex consisting of basic forms such that the projection of the support to the space of leaf closures $M/\overline{\F}$ is compact, which was introduced by Sergiescu \cite{Sergiescu1985} to study Poincar\'{e} type duality for Riemannian foliations. The following is a result on the duality between reduced Haefliger cohomology and basic cohomology.

\begin{thm}\label{thm:dfol}
For a codimension $n$ Riemannian foliation $\F$ on a manifold $M$ with a complete flat vector bundle $(\E,\nabla)$, we have 
\begin{align}
\widehat{H}^{\bt}_{c}(\Tr \F;\E) & \cong H^{n-\bt}(M,\F;\E^{*} \otimes \Om_{\nu})^{*}, \label{eq:dual1} \\
\widehat{H}^{\bt}_{lc}(\Tr \F;\E) & \cong H^{n-\bt}_{c}(M,\F;\E^{*} \otimes \Om_{\nu})^{*}, \notag 
\end{align} 
where $\Om_{\nu}$ is the orientation bundle of $TM/T\F$.
\end{thm}
In the case where $\E$ is trivial, \eqref{eq:dual1} is known~\cite[p.~598]{AlvarezLopezMasa2008}. In the case where $M$ is compact, more general duality for the whole spectral sequence associated to Riemannian foliations was proven by \'{A}lvarez L\'{o}pez~\cite{AlvarezLopez1989b} and Dom\'{i}nguez \cite{Dominguez1998}. As mentioned above, it is well known that Haefliger cohomology is a part of the spectral sequence associated with $\F$.

Theorems~\ref{thm:masa} and~\ref{thm:dfol} imply the duality between Haefliger cohomology and basic cohomology of Riemannian foliations. Combining this duality with known results on basic cohomology, we can compute $H^{\bt}_{c}(\Tr \F;\E)$ as follows. Let $\P$ be the Sergiescu's orientation sheaf of $(M,\F)$, which was introduced in~\cite[Section~1]{Sergiescu1985}. Here $\P = (\wedge^{\rank \mathcal{C}} \mathcal{C})^{+} \otimes \mathcal{O_{\nu}}$, where $(\wedge^{\rank \mathcal{C}} \mathcal{C})^{+}$ is the flat real line bundle obtained by taking the absolute value of the determinant of the Molino's commuting sheaf. Note that,  since $\P$ has basic holonomy (see \cite{AlvarezLopez1996}), if $\F$ is complete, then so is $\P$.

\begin{cor}\label{cor:inv}
For a complete codimension $n$ Riemannian foliation $\F$ on a manifold $M$ with a flat vector bundle $(\E,\nabla)$, we have the following.
\begin{enumerate}
\item If $\E$ is complete, then we have $H^{\bt}_{c}(\Tr \F;\E) \cong H^{n-\bt}_{lc}(\Tr \F; \E^{*} \otimes \P)^{*}$.
\item $H^{0}_{lc}(\Tr \F; \P) \cong \RR$. 
\item $H^{0}_{c}(\Tr \F; \P) \cong \RR$ if and only if $M/\overline{\F}$ is compact.
\end{enumerate}
\end{cor}
Corollary~\ref{cor:inv} was first proven by Hector~\cite{Hector1988} in the case where $M$ is compact, $\E$ is trivial and the structure Lie algebra of $\F$ is compact or nilpotent. Masa~\cite[Duality theorem]{Masa1992} proved it in the case where $M$ is compact and $\E$ and $\P$ are trivial. Dom\'{i}nguez~\cite[Theorem~5.2]{Dominguez1998} proved it in the case where $M$ is compact and $\E$ and $\P$ may be nontrivial.

Note that $\P$ is topologically trivial if and only if $\F$ is transversely orientable. If $\P$ is topologically trivial, then there exists a flat real line bundle $\P^{1/2}$ such that $\P^{1/2} \otimes \P^{1/2} \cong \P$. We have the following direct consequence of (ii) of the last corollary.

\begin{cor}
For a connected manifold $M$ with a transversely orientable codimension $n$ complete Riemannian foliation $\F$ whose space of leaf closures is compact, we have $H^{\bt}_{c}(\Tr \F;\P^{1/2}) \cong H^{n-\bt}_{c}(\Tr \F;\P^{1/2})^{*}$.
\end{cor}

The basic cohomology of Riemannian foliations is a topological invariant by El Kacimi Alaoui-Nicolau \cite{ElKacimiAlaouiNicolau1993}. This result was extended to invariant cohomology of Riemannian pseudogroups by \'{A}lvarez L\'{o}pez-Masa \cite[Corollary 27.3 and Theorem 29.1]{AlvarezLopezMasa2008}. Combining their results with Theorems \ref{thm:masa} and \ref{thm:dfol}, we have the following.
\begin{cor}
$H^{\bt}_{c}(\Tr \F; \E)$ and $H^{\bt}_{lc}(\Tr \F; \E)$ are topological invariants for Riemannian foliations with complete flat vector bundles. Namely, these cohomologies are isomorphic for a Riemannian foliation $(M_{i},\F_{i})$ with a complete flat vector bundle $\E_{i}$ ($i=1,2$), if there exists a homeomorphism $\varphi : (M_{1},\F_{1}) \to (M_{2},\F_{2})$ such that $h_{2} \circ \varphi_{*} =  h_{1}$, where $h_{i}$ is the holonomy homomorphism of $\E_{i}$ $(i=1,2)$.
\end{cor}

\subsection{Tenseness theorems}\label{sec:I3}

By consequences of Theorems~\ref{thm:domi}, \ref{thm:masa} and \ref{thm:dfol}, we will prove the following result, which is a generalization of tenseness theorem of Dom\'{i}nguez~\cite{Dominguez1998} for Riemannian foliations on closed manifolds. 

\begin{cor}\label{cor:st}
Any complete Riemannian foliation is strongly tense.
\end{cor}

Since any transversely complete Riemannian foliation is complete \cite[Proposition 3.1.2]{Haefliger1985}, Corollary \ref{cor:st} is a generalization of \cite[Theorem 1.2]{NozawaRoyoPrieto2014}, which asserts that any transversely complete Riemannian foliation of dimension one is strongly tense.

If $M$ is connected and $\F$ is transversely orientable, then $\P$ is a trivial flat real line bundle if and only if $H^{n}_{c}(M,\F) \cong \RR$ by a result of Sergiescu \cite{Sergiescu1985}. We show the following generalization of a result of \'{A}lvarez L\'{o}pez \cite{AlvarezLopez1996} and one direction of Masa's characterization~\cite[Minimality theorem]{Masa1992} of taut Riemannian foliations. 

\begin{cor}\label{cor:masa}
Let $h$ be the holonomy homomorphism of the Sergiescu's orientation sheaf $\P$ of a complete Riemannian foliation $\F$. If $h(\pi_{1}M) \subset \{\pm 1\}$, then $\F$ is taut. In particular, if $\P$ is trivial as a flat real line bundle, then $\F$ is taut.
\end{cor}

There are simple examples of taut Riemannian foliations on open manifolds whose Sergiescu's orientation sheaf is nontrivial (see Example~\ref{ex:con}). So the other direction of Masa's characterization does not hold in general for Riemannian foliations on open manifolds.

\subsection{The case of Lie foliations}\label{sec:exam}

We will illustrate our main results with Lie foliations. Let $G$ be a simply connected Lie group. A foliation on a closed manifold is a $G$-Lie foliation if and only if its holonomy pseudogroup is equivalent to the pseudogroup $\H_{\Gamma}$ on $G$ generated by the left action of a subgroup $\Gamma$ of $G$. A $G$-Lie foliation is always Riemannian. Here we consider a complete $G$-Lie foliation $\F$ on a manifold $M$. Its holonomy pseudogroup is equivalent to $\H_{\Gamma}$ for a subgroup $\Gamma$ of $G$, and by \cite{Fedida1972}, we obtain the following data: 
\begin{itemize}
\item a homomorphism $\hol : \pi_{1}M \to G$ such that $\hol(\pi_{1}M)=\Gamma$ and 
\item a submersion $\dev : \widetilde{M} \to G$
\end{itemize}
with an equivariant condition $\dev(\gamma \cdot x)=  \hol (\gamma) \dev (x)$ for $\forall x \in \widetilde{M}$, $\forall \gamma \in \pi_{1}M$, where $\widetilde{M}$ is the universal cover of $M$. It is easy to see that we can recover the foliation $\F$ by taking quotient of the foliation on $\widetilde{M}$ whose leaves are the fibers of $\dev$. As we will see in Sections \ref{sec:4} and \ref{sec:6} of this paper, the computation of the Haefliger cohomology of complete Riemannian foliations can be reduced to the case of minimal $G$-Lie foliations. The group $\Gamma = \hol(\pi_{1}M)$ is called the {\em holonomy group} of $(M,\F)$. It is easy to see that $\F$ is minimal if and only if its holonomy group is dense in $G$. 

Given a finitely generated dense subgroup $\Gamma$ of a Lie group $G$, we can construct a minimal complete $G$-Lie foliation $(M,\F)$ with holonomy group $\Gamma$, for example, by suspension. Let $\rho : G \to \GL(V)$ be a $G$-representation on a vector space $V$. Then we can construct a flat vector bundle $\E_{V}$ on $M$ as follows: Consider a flat vector bundle $\E_{G}$ over $G$ defined by the first projection $G \times V \to G$ with parallel sections of the form $G \to G \times V; g \mapsto (g, \rho(g^{-1})v)$ for some $v \in V$. Then, we obtain a flat vector bundle over $M$ by taking quotient of the flat vector bundle $\dev^{*}\E_{G}$ over the universal cover $\widetilde{M}$ by the $\pi_{1}M$-action given by the holonomy homomorphism. In this case, the holonomy pseudogroup $\Hol_{T}(\F,\E)$ is equivalent to the pseudogroup $\widetilde{\H}_{\Gamma}$ on $G \times V$ generated by the $\Gamma$-action given by $\gamma (g,v) = (\gamma g, v)$ for $\forall g \in G$, $\forall v \in V$ and $\forall \gamma \in \Gamma$. Since $\Gamma$ is dense in $G$, a $k$-form $\alpha$ on $G$ with values in $\E_{G}$ is $\widetilde{\H}_{\Gamma}$-invariant if and only if $\alpha$ is $G$-invariant. Here, by construction, $\alpha$ is $G$-invariant if and only if $\alpha$ is parallel. This implies that there is a one-to-one correspondence between $\Gamma$-invariant $k$-forms valued in $\E_{V}$ and the fiber $\wedge^{k} \mg^{*} \otimes V$ over the unit element of $G$. Thus we have  
\begin{align}
\Omega^{\bt}(M,\F;\E_{V}) & \cong \wedge^{\bt} \mg^{*} \otimes V\;, \label{eq:co0} \\
H^{\bt}(M,\F;\E_{V}) & \cong H^{\bt}(\mg;V)\;, \notag 
\end{align}
where $\mg = \Lie(G)$. Note that, since the leaves of $\F$ is dense in $M$, we have $H^{\bt}_{c}(\Tr \F;\E_{V}) \cong H^{\bt}_{lc}(\Tr \F;\E_{V})$. Since $H^{\bt}(\mg; V^{*})^{*} \cong H_{\bt}(\mg; V)$ and $H^{\bt}(\mg; V^{*})^{*} \cong H^{\dim G - \bt}(\mg; V \otimes \det \mg)$ by duality of Lie algebra cohomology  (see \cite[Theorem~6.10]{Knapp1988}), as a special case of consequence of Theorems \ref{thm:masa} and \ref{thm:dfol}, we obtain the following. 
\begin{cor}\label{cor:Lie}
Let $(M,\F)$ be a transversely oriented complete $G$-Lie foliation. Let $\E_{V}$ be the flat vector bundle over $M$ constructed above with a $G$-representation $V$. Then we have 
\[
H^{\bt}_{c}(\Tr \F;\E_{V}) \cong H^{\bt}_{lc}(\Tr \F;\E_{V}) \cong H_{\dim G - \bt}(\mg; V) \cong H^{\bt}(\mg; V \otimes \det \mg)^{*}\,.
\]
\end{cor}
Corollary~\ref{cor:inv} in this case directly follows from Corollary~\ref{cor:Lie} and the Poincar\'{e} duality of Lie algebra cohomology (see~\cite[Theorem~6.10]{Knapp1988}). In this case, $\P$ is isomorphic to the flat vector bundle $\E_{\det \mg^{-}}$, where $\mg^{-}$ is the Lie algebra of right invariant vector fields on $G$. As a special case of Corollaries~\ref{cor:st} and~\ref{cor:masa}, we get the following.
\begin{cor}\label{cor:stu}
Any minimal complete $G$-Lie foliation is strongly tense. Any minimal complete $G$-Lie foliation is taut if $G$ is unimodular.
\end{cor}
In the case where $M$ is compact, Corollaries~\ref{cor:Lie} and~\ref{cor:stu} are special cases of theorems of Masa and Dom\'{i}nguez. Note that if $M$ is compact, then the holonomy group $\Gamma$ is necessary to be compactly generated in the sense of Haefliger \cite{Haefliger2002} (see also \cite{Meigniez1997}). In these corollaries, we can compute Haefliger cohomology for minimal complete Lie foliations in general.

\begin{ack}
The author is very grateful to Gilbert Hector for stimulative discussions. Lemma~\ref{lem:lied} in the case of trivial coefficient is due to him. The author is very grateful to Jes\'{u}s Antonio \'{A}lvarez L\'{o}pez for helpful discussions and many valuable comments on the first manuscript. The author is very grateful to Xos\'{e} Maria Masa V\'{a}zquez and Jos\'{e} Ignacio Royo Prieto for helpful discussions. The author is very grateful to anonymous referee to point out errors in the manuscript. This paper was written during the stay of the author at Institut des Hautes \'{E}tudes Scientifiques, Institut Mittag-Leffler and Centre de Recerca Matem\`{a}tica and University of Santiago de Compostela; he is very grateful for their hospitality. 
\end{ack}

\section{Rummler-Sullivan type characterization of strong tenseness}

Let $(M,g)$ be a Riemannian manifold with a $p$-dimensional foliation $\F$. In the case where $\F$ is oriented, the {\em characteristic form} of $(M,\F,g)$ is defined to be the unique $p$-form $\chi$ on $M$ such that $\ker \chi = (T\F)^{\perp}$ and $\chi \big|_{\bigwedge^{p}T\F}$ is the oriented leafwise volume form of norm one at each point on $M$. In general, on any simply connected open subset $U$ of $M$, we have two characteristic forms $\chi$ of $(U,\F|_{U},g|_{U})$ depending the choice of an orientation of $\F|_{U}$. We call such $\chi$ a {\em local characteristic form} of $(M,\F,g)$. By Rummler's formula~\cite{Rummler1979}, the mean curvature form $\kappa$ of $(M,\F,g)$ is determined by a local characteristic form $\chi$ as follows:
\begin{equation}\label{eq:R}
\kappa(Y)=-d\chi(Y, E_{1}, \ldots, E_{p})\;,
\end{equation}
for any local vector field $Y$, where $\{E_1,\dots,E_p\}$ is a local orthonormal frame of $T\F$ such that $\chi (E_{1}, \ldots, E_{p}) = 1$.

Let $(\E,\nabla)$ be a flat vector bundle over $M$. We recall the following terminologies.
\begin{defn}
We call $\omega\in \Omega^{\bt}(M;\E)$ an $\F$-{\em trivial} form if $\iota(X_{1}) \circ \iota(X_{2})  \circ \cdots \circ \iota(X_{p})\omega = 0$ for any vector fields $X_{1}$, $X_{2}$, $\ldots$, $X_{p}$ tangent to $\F$. We call $\omega\in \Omega^{\bt}(M;\E)$ {\em relatively} $\F$-{\em closed} if $d_{\nabla}\omega$ is $\F$-trivial, where $d_{\nabla}$ is the covariant derivative.
\end{defn}
Rummler's formula \eqref{eq:R} is expressed in terms of $\F$-trivial forms as follows:
\begin{lem}\label{lem:R}
For a foliated Riemannian manifold $(M,\F,g)$ with characteristic form $\chi$ and mean curvature form $\kappa$, there exists an $\F$-trivial form $\beta$ such that
\begin{equation}\label{eq:R2}
d\chi = -\kappa \wedge \chi + \beta\;.
\end{equation}
\end{lem}

For a flat real line bundle $\E$ over $M$ with holonomy homomorphism $h : \pi_{1}M \to \Aut (\RR) \cong \RR^{\times}$, let $(\E^{+},\nabla^{+})$ be the topologically trivial line bundle over $M$ whose holonomy homomorphism is the absolute value $|h|$ of $h$. 

\begin{prop}\label{prop:tense}
Let $(\E,\nabla)$ be a flat real line bundle over $M$ with holonomy homomorphism $h$. For a foliated Riemannian manifold $(M,\F,g)$, the following are equivalent:
\begin{enumerate}
\item The mean curvature form $\kappa$ of $(M,\F,g)$ is a closed basic $1$-form such that $\log |h(\gamma)| = -\int_{\gamma}\kappa$ for $\forall \gamma \in \pi_{1}M$.
\item There exists a nowhere vanishing section $\sigma$ of $\E^{+}$ such that $\sigma|_{L}$ is parallel for every leaf $L$ of $\F$ and $\chi \otimes \sigma$ is relatively $\F$-closed for any local characteristic form $\chi$ of $(M,\F,g)$.
\end{enumerate}
\end{prop}

\begin{proof}
Let $\sigma$ be a nowhere vanishing section of $\E^{+}$ such that $\sigma|_{L}$ is parallel for every leaf $L$ of $\F$. On a connected simply-connected open subset $U$ of $M$, we take a nowhere vanishing function $f$ on $U$ so that $\frac{1}{f}\sigma$ is a parallel section of $(\E^{+}|_{U},\nabla^{+}|_{U})$. Note that $f$ is basic. We get
\begin{equation}\label{eq:g1}
d_{\nabla}\sigma = df \wedge \frac{1}{f} \sigma = d\log |f| \wedge \sigma\;.
\end{equation}
Let $p=\dim \F$. By Rummler's formula (Lemma~\ref{lem:R}), for any of two local characteristic forms $\chi$ of $(M,\F,g)$ defined on $U$, there exists an $\F$-trivial form $\beta$ which satisfies $d\chi = -\kappa \wedge \chi + \beta$. Then, by~\eqref{eq:g1}, we get
\begin{align}
d_{\nabla}(\chi \otimes \sigma) & = d\chi \otimes \sigma + (-1)^{p} \chi \wedge d_{\nabla}\sigma \notag \\
& = \big(d\chi + (-1)^{p} \chi \wedge d\log |f| \big) \otimes \sigma \label{eq:compt} \\
& = (-\kappa + d\log |f|) \wedge \chi \otimes \sigma + \beta \otimes \sigma\;. \notag
\end{align}

We show that (ii) implies (i). Since $\chi \otimes \sigma$ is relatively $\F$-closed, then, by \eqref{eq:compt}, we get $-\kappa + d\log |f|=0$, which implies that $\kappa$ is a closed basic $1$-form such that $\log |h(\gamma)| = -\int_{\gamma}\kappa$ for $\forall \gamma \in \pi_{1}M$.

We show that (i) implies (ii). Assume that the mean curvature form $\kappa$ is a closed basic $1$-form such that $\log |h(\gamma)| = -\int_{\gamma}\kappa$ for $\forall \gamma \in \pi_{1}M$. Then $\E^{+}$ is obtained by the quotient of $\widetilde{M} \times \RR$ by the diagonal $\pi_{1}M$-action $\rho$ defined by $\rho(\gamma) (x, v) = (\gamma \cdot x, |h(\gamma)| v )$ for $\gamma \in \pi_{1}M$, $x \in \widetilde{M}$ and $v \in \RR$.  Take a primitive $\zeta$ of $u^{*}\kappa$, where $u : \widetilde{M} \to M$ is the universal cover. Then the graph of $e^{-\zeta}$ in $\widetilde{M} \times \RR$ is $\rho$-invariant, and its quotient in $\E$ is the image of a nowhere vanishing global section $\sigma$ of $\E$. Since $\kappa$ is basic, $\sigma|_{L}$ is parallel for every leaf $L$ of $\F$. Since $e^{\zeta}\sigma$ is parallel on simply connected open subset in $M$, we can take $f(x) = e^{\zeta(x)}$ in \eqref{eq:compt}. It follows that $d_{\nabla}(\chi \otimes \sigma)$ is $\F$-trivial. Then $\chi \otimes \sigma$ is relatively $\F$-closed, which concludes the proof.
\end{proof}

Let us define the following.
\begin{defn}\label{defn:stE}
For a foliated manifold $(M,\F)$ and  a flat real line bundle $(\E,\nabla)$ over $M$, we say that $(M,\F)$ is {\em strongly tense with} $(\E,\nabla)$ if there exists a Riemannian metric $g$ such that the mean curvature form $\kappa$ of $(M,\F,g)$ is a closed basic $1$-form such that $\log |h(\gamma)| = -\int_{\gamma}\kappa$ for $\forall \gamma \in \pi_{1}M$.
\end{defn}

Since we take the absolute value of the holonomy map of $(\E,\nabla)$ in this definition, $(M,\F)$ is strongly tense with respect to $(\E,\nabla)$ if and only if $(M,\F)$ is strongly tense with respect to $(\E^{+},\nabla^{+})$. Note that, as we see in the proof of the last proposition, the mean curvature depends only on the absolute value of the holonomy of the flat real line bundles. Let us show the following Rummler-Sullivan type characterization by using the last proposition and Sullivan's purification.
\begin{prop}\label{prop:RS}
For a closed manifold $M$ with a $p$-dimensional orientable foliation $\F$ and a topologically trivial flat real line bundle $(\E,\nabla)$ with basic holonomy over $(M,\F)$, the following are equivalent:
\begin{enumerate}
\item $(M,\F)$ is strongly tense with respect to $(\E,\nabla)$.
\item There exists a relatively $\F$-closed form $\eta \in \Omega^{p}(M;\E)$ whose restriction to any leaf is nowhere vanishing.
\end{enumerate}
\end{prop}

\begin{proof}
(ii) follows from (i) by the last proposition. Let us show that (i) follows from (ii) by using Sullivan's purification. Assume that there exists a relatively $\F$-closed form $\eta \in \Omega^{p}(M;\E)$ whose restriction to any leaf is nowhere vanishing. Take a section $\sigma$ of $\E$ whose restriction to each leaf of $\F$ is parallel as in the proof of the last proposition. Then we have $\eta = \chi \otimes \sigma$ for some $\chi \in \Omega^{p}(M)$ whose restriction to each leaf is a volume form. Here $\chi$ may not be a characteristic form, and we modify it by Sullivan's purification. Let $\varpi : C^{\infty}(TM) \to C^{\infty}(T\F)$ be the $C^{\infty}$-linear projection map determined by 
\[
\iota \big(\varpi(X)\big) \big(\chi \big|_{\bigwedge^{p}T\F} \big) = \big( \iota(X)\chi \big) \big|_{\bigwedge^{p-1}T\F}\;,
\]
which is well-defined by the fact that $\chi \big|_{\bigwedge^{p}T\F}$ is a leafwise volume form of $\F$. Let $\tilde{\eta} = \varpi^{*}\eta$. It is easy to see that $\ker \tilde{\eta}$ is a plane field complement to $T\F$ and hence $\tilde{\eta}$ is a characteristic form on $(M,\F)$ with some Riemannian metric. We show that $\F$ satisfies the condition of (ii) of the last proposition. For that, it suffices to show that $\tilde{\eta}$ is relatively $\F$-closed. By definition of the purification, we have $\tilde{\eta} (X_{1}, \cdots, X_{p-1}, Y) = \eta (X_{1}, \cdots, X_{p-1}, Y)$ for $X_{1}$, $\ldots$, $X_{p-1} \in C^{\infty}(T\F)$ and $Y \in C^{\infty}(TM)$, from which relatively $\F$-closedness of $\tilde{\eta}$ follows. Since $\E$ admits a nowhere vanishing leafwise parallel section by the argument in the last paragraph of the proof of Proposition \ref{prop:tense}, $\F$ satisfies the condition (ii) of the last proposition, which implies that $\F$ is strongly tense with respect to $(\E,\nabla)$.
\end{proof}

\section{Haefliger type characterization of strong tenseness}\label{sec:2}

\subsection{Haefliger cohomology of linear pseudogroups on flat vector bundles}\label{sec:Hcoh}

Recall that a {\em pseudogroup} on a manifold $N$ is a collection of diffeomorphisms between two open sets of $N$ which contains the identities and is closed under composition, inversion, restriction and union. In this article, we consider flat vector bundles with pseudogroup actions in the following sense, which is a variant of the one used in~\cite[Section~3.2]{Haefliger1985}.

\begin{defn}\label{defn:linearH}
Let $\pi : (\E,\nabla) \to T$ be a flat vector bundle over a manifold $T$. We say that a pseudogroup $\H$ on $\E$ is {\em linear}, or $\H$ {\em linearly acts on} $(\E,\nabla)$ if $\H$ is generated by local diffeomorphisms $h$ on $\E$ of the following form: there exist open sets $U$, $V$ of $T$ and a diffeomorphism $h_{T} : U \to V$ such that  
\begin{enumerate} 
\item $\Dom h = \pi^{-1}(U)$, $\image h = \pi^{-1}(V)$ and
\item $h$ is an isomorphism of flat vector bundles $(\E,\nabla)|_{U} \to (\E,\nabla)|_{V}$ which covers $h_{T} : U \to V$.
\end{enumerate}
The pseudogroup $\H_{T}$ on $T$ generated by the above $h_{T}$ is called the {\em base pseudogroup} induced by $\H$. For $h \in H$ and a differential form $\alpha$ with values in $\E$ such that $\pi^{-1}(\supp \alpha) \subset \image h$, we will denote the extention by zero of the pull back $h^{-1} \circ \alpha \circ h_{T}$ of $\alpha$ by $h^{*}\alpha$.
\end{defn}

In the case where there is a one to one correspondence between $\H$ and $\H_{T}$, we may identify $\H$ with $\H_{T}$ and we say that $\H_{T}$ linearly acts on $\E$.

\begin{rem}
In the situation of Definition \ref{defn:linearH}, the set consisting of all germs of the elements of $\H$ is a groupoid, whose objects are the points of $\E$. This groupoid acts on $\E$, and the groupoid obtained from the germs of the base pseudogroup acts on $T$.
\end{rem}

A typical example is obtained from a foliated manifold $(M,\F)$ with a flat vector bundle $(\E,\nabla)$ and a total transversal $T$. Consider the pseudogroup $\Hol_{T}(\F,\E)$ on $\E|_{T}$ introduced in Section \ref{sec:I2}, which is generated by holonomy maps on $\E$ along leaf paths of $\F$ which connects two points on $T$. Then $\Hol_{T}(\F,\E)$ is a linear pseudogroup on $\E|_{T}$ whose base pseudogroup is the holonomy pseudogroup of $\F$ on $T$.

Consider a flat vector bundle $\pi : (\E,\nabla) \to T$ with a linear pseudogroup $\H$. Let $\Omega_{c}^{\bt}(T;\E)$ be the space of compactly supported differential forms on $T$ with values in $\E$, and $\L^{\bt}_{c, \H}$ the subspace of $\Omega_{c}^{\bt}(T;\E)$ consisting of finite sums of elements of the form $h^{*}\alpha - \alpha$ for $h\in \H$ and $\alpha\in \Omega_{c}^{\bt}(T;\E)$ such that $\pi^{-1} (\supp \alpha) \subset \image h$. Let
\[
\Omega_{c}^{\bt}(\E/\H) := \Omega_{c}^{\bt}(T;\E)/\L^{\bt}_{c, \H}\;.
\]
Since $\nabla$ is $\H$-invariant, the covariant derivative $d_{\nabla} : \Omega_{c}^{\bt}(T;\E) \to \Omega_{c}^{\bt +1}(T;\E)$ commutes with $\H$. Thus we get $d_{\nabla} : \Omega_{c}^{\bt}(\E/\H) \to \Omega_{c}^{\bt+1}(\E/\H)$.

\begin{defn}
The cohomology of the complex $(\Omega_{c}^{\bt}(\E/\H), d_{\nabla})$ is called the {\em Haefliger cohomology} of $\H$ {\em with values in} $(\E,\nabla)$ and denoted by $H^{\bt}_{c}(\E/\H)$.
\end{defn}

We topologize $\Omega_{c}^{\bt}(T;\E)$ with the LF-topology (see~\cite[Section~9]{deRham1984}). Since $\L^{\bt}_{c, \H}$ is not closed in $\Omega_{c}^{\bt}(T;\E)$ in general, the complex $\Omega_{c}^{\bt}(\E/\H)$ may not be Hausdorff with the quotient topology. The quotient of $\Omega_{c}^{\bt}(\E/\H)$ by the closure of $0$ is denoted by $\widehat{\Omega}^{\bt}_{c}(\E/\H)$ and its cohomology is denoted by $\widehat{H}^{\bt}_{c}(\E/\H)$. Following~\cite{AlvarezLopezMasa2008}, we call $\widehat{H}^{\bt}_{c}(\E/\H)$ the {\em reduced Haefliger cohomology of} $\H$ {\em with values in} $\E$. We have a canonical map
\[
\tau : H^{\bt}_{c}(\E/\H) \longrightarrow \widehat{H}^{\bt}_{c}(\E/\H)\;.
\]

Haefliger~\cite{Haefliger2002} defined morphisms and equivalences between pseudogroups. Let $\H_{i}$ be a pseudogroup on a manifold $T_{i}$ $(i=1,2)$. As noted in a preprint \cite{Raimbaud2009}, we can see that a set $\Phi = \{\Phi^{j}\}$ of diffeomorphisms from an open set of $T_{1}$ to an open set of $T_{2}$ is an equivalence from $\H_{1}$ to $\H_{2}$ if and only if $\Phi$ is a maximal set such that 
\begin{itemize}
\item $T_{1} \subset \cup_{j} \Dom \Phi^{j}$, $T_{2} \subset \cup_{j} \image \Phi^{j}$,
\item  $\Phi \H_{1} \Phi^{-1} \subset \H_{2}$ and $\Phi^{-1} \H_{2} \Phi \subset \H_{1}$.
\end{itemize}
Haefliger cohomology with values in flat vector bundles is invariant under the equivalence in the following sense: Consider a linear pseudogroup $\H_{i}$ on a flat vector bundle $(\E_{i},\nabla_{i}) \to T_{i}$. If there exist an open covering $\{V^{j}\}$ of $T_{1}$, an open covering $\{W^{j}\}$ of $T_{2}$ and bundle maps $\Phi^{j} : \E_{1}|_{V^{j}} \to \E_{2}|_{W^{j}}$ such that  
\begin{enumerate}
\item $\Phi^{j}$ is an isomorphism of flat vector bundles $(\E_{1},\nabla_{1})|_{V^{j}} \to (\E_{2},\nabla_{2})|_{W^{j}}$  and
\item $\{\Phi^{j}\}$ generates an equivalence between $\H_{1}$ and $\H_{2}$ in the sense of Haefliger,
\end{enumerate}
then we have $\Omega^{\bt}_{c}(\E_{1}/\H_{1}) \cong \Omega^{\bt}_{c}(\E_{2}/\H_{2})$ and,  in particular, $H^{\bt}_{c}(\E_{1}/\H_{1}) \cong H^{\bt}_{c}(\E_{2}/\H_{2})$. In~\cite[Theorem~2.5]{Raimbaud2009}, it is shown that this equivalence of pseudogroups is equivalent to the Morita equivalence of the groupoids obtained by taking germs the pseudogroups. In this sense, Haefliger cohomology can be regarded as invariants of groupoids.

\subsection{Characterization of strong tenseness in terms of Haefliger cohomology}\label{sec:charst}

Let $M$ be a manifold with a foliation $\F$ and a flat vector bundle $(\E,\nabla)$. A {\em transversal} of $(M,\F)$ is an immersion $T \to M$ which is transverse to the leaves of $\F$ by definition, but throughout this article we regard $T$ as a subset of $M$. Recall that a transversal $T$ of $(M,\F)$ is called {\em total} if it intersects every leaf of $\F$. Take a total transversal $T$ and consider the holonomy pseudogroup $\Hol_{T}(\F,\E)$ on $\E|_{T}$ generated by the lift of holonomy maps of $\F$ to $\E$ (see Section \ref{sec:I2}), which acts on $\E|_{T}$. Since the equivalence class of $\Hol_{T}(\F,\E)$ is independent of the choice of $T$, so is the isomorphism class of $\Omega^{\bt}_{c}\big(\E|_{T}/\Hol_{T}(\F,\E)\big)$. Thus the Haefliger cohomology of $(M,\F)$ with values in $(\E,\nabla)$ is well-defined as the cohomology of this complex. 

\begin{defn}
For a foliated manifold $(M,\F)$ with a total transversal $T$ and a flat vector bundle $(\E,\nabla)$, let 
\begin{align*}
\Omega^{\bt}_{c}(\Tr \F;\E) & := \Omega^{\bt}_{c}\big(\E|_{T}/\Hol_{T}(\F,\E)\big)\;, \\
H^{\bt}_{c}(\Tr \F;\E) & := H^{\bt}\big(\Omega^{\bt}_{c}\big(\E|_{T}/\Hol_{T}(\F,\E)\big), d_{\nabla}\big)\;
\end{align*}
and $H^{\bt}_{c}(\Tr \F;\E)$ is called the {\em Haefliger cohomology} of $(M,\F)$ {\em with values in} $(\E,\nabla)$.
\end{defn}

The reduced complex $\widehat{\Omega}_{c}^{\bt}\big(\E|_{T}/\Hol_{T}(\F,\E)\big)$ is the quotient of $\Omega_{c}^{\bt}\big(\E|_{T}/\Hol_{T}(\F,\E)\big)$ by the closure of zero. The complex $\widehat{\Omega}^{\bt}_{c}(\Tr \F;\E):= \widehat{\Omega}_{c}^{\bt}\big(\E|_{T}/\Hol_{T}(\F,\E)\big)$ and the reduced Haefliger cohomology $\widehat{H}^{\bt}_{c}(\Tr \F;\E):= H^{\bt}(\widehat{\Omega}^{\bt}_{c}(\Tr \F;\E))$ are well-defined for $(M,\F)$ and $\E$ being independent of the choice of $T$.

For open foliated manifolds, we cannot use Haefliger cohomology to characterize the tautness. For example, any characteristic form of an open foliated manifold is not compactly supported. Here we introduce a variant of Haefliger cohomology which is useful to consider tautness and strong tenseness of complete Riemannian foliations on possibly open manifolds. For a general linear pseudogroup $\H$ on a flat vector bundle $\E$ over a manifold $T$, let 
\begin{multline*}
\Omega^{\bt}_{lc, \H}(T;\E) = \{ \, \alpha \in \Omega^{\bt}(T;\E) \mid \forall x \in T, \\
\exists K: \text{ a compact neighborhood of } x \text{ in } T \text{ s.t. } \alpha|_{\H K} \in \Omega_{c}^{\bt}(\H K;\E|_{\H K}) \, \} \end{multline*}
and $\Lambda^{\bt}_{lc, \H}$ be a subcomplex of $\Omega^{\bt}_{lc, \H}$ which consists of $\alpha \in \Omega^{\bt}_{lc, \H}(T;\E)$ of the form $\alpha = \sum_{i \in I} h_{i}^{*}\alpha_{i} - \alpha_{i}$ for $\alpha_{i} \in \Omega^{\bt}_{lc, \H}(T;\E)$ and $h_{i} \in \H$ such that 
\begin{itemize}
\item $\pi^{-1}(\supp \alpha_{i}) \subset \image h_{i}$, where $\pi : \E \to T$ is the projection, and
\item $\{ \supp \alpha_{i}, (h_{i}^{-1})_{T}(\supp \alpha_{i}) \}_{i \in I}$ is locally finite, where $(h_{i}^{-1})_{T}$ is the map induced on $T$ by $h_{i}^{-1}$.
\end{itemize}
The topology of $\Omega^{\bt}_{lc, \H}(T;\E)$ is defined as follows: A sequence $\{\alpha_{j}\}$ in $\Omega^{\bt}_{lc, \H}(T;\E)$ converges to $\alpha \in \Omega^{\bt}_{lc, \H}(T;\E)$ if, for any $x \in T$, there exists a compact neighborhood $K$ of $x$ in $T$ such that  $\alpha|_{\H K}$, $\alpha_{k}|_{\H K} \in \Omega^{\bt}_{c}(\H K; \E|_{\H K})$ for any $k \in \ZZ_{>0}$ and $\lim_{k \to \infty} \alpha_{k}|_{\H K} = \alpha|_{\H K}$ in $\Omega^{\bt}_{c}(\H K; \E|_{\H K})$. Now we define the leafwise compactly supported Haefliger cohomology as follows:
\begin{align*}
\Omega^{\bt}_{lc}(\E/\H) & := \Omega^{\bt}_{lc, \H}(T;\E) / \Lambda^{\bt}_{lc, \H}\;, \\
H^{\bt}_{lc}(\E/\H) & := H^{\bt}\big(\Omega^{\bt}_{lc}(\E/\H), d_{\nabla}\big)\;.
\end{align*}
The leafwise compactly supported Haefliger cohomology $H_{lc}^{\bt}(\Tr \F;\E)$ of a foliated manifold $(M,\F)$ with values in $\E$ is defined in terms of the holonomy pseudogroup $\Hol_{T}(\F,\E)$ on a total transversal $T$ as the cohomology of the complex $\Omega^{\bt}_{lc}(\Tr \F;\E) := \Omega^{\bt}_{lc}(\E|_{T}/\Hol_{T}(\F,\E))$. The reduced versions $\widehat{H}^{\bt}_{lc}(\E/\H)$ and $\widehat{H}^{\bt}_{lc}(\Tr \F;\E)$ are defined as the cohomology of the quotient of the complexes by the closure of zero.

It is easy to see that, if the space of orbit closures $T/\overline{\H_{T}}$ of the base pseudogroup $\H_{T}$ induced by $\H$ is compact, then $\Omega^{\bt}_{lc}(\E/\H) = \Omega^{\bt}_{c}(\E/\H)$ and hence $H^{\bt}_{lc}(\E/\H) \cong H^{\bt}_{c}(\E/\H)$. In particular, if the space of leaf closures of a foliated manifold $(M,\F)$ is compact, we have $\Omega^{\bt}_{lc}(\Tr \F;\E) \cong \Omega^{\bt}_{c}(\Tr \F; \E)$ and $H^{\bt}_{lc}(\Tr \F;\E) \cong H^{\bt}_{c}(\Tr \F; \E)$ for any flat vector bundle $\E$ over $M$.

We state a Haefliger type characterization of strong tenseness. Recall that we call a section $\xi$ of a topologically trivial real line bundle $\E$ {\em nonnegative} if $\xi$ is identified with a nonnegative function under a topological trivialization of $\E$. For such section $\zeta$, let $\supp_{+}^{T}\zeta := \{ x \in T \mid \zeta(x) \neq 0 \}$.

\begin{thm}\label{thm:Hchar}
For a closed manifold $M$ with an oriented foliation $\F$ and a topologically trivial flat real line bundle $(\E,\nabla)$ with basic holonomy over $(M,\F)$, the following are equivalent:
\begin{enumerate}
\item $(M,\F)$ is strongly tense with $(\E,\nabla)$. 
\item There exist a total transversal $T$ of $(M,\F)$ and a nonnegative cocycle $\xi \in \Omega^{0}_{c}(\E|_{T}/\Hol_{T}(\F,\E))$ such that $\supp^{\F}_{+} \xi=M$.
\item For any total transversal $T$ of $(M,\F)$ and any compact set $K$ in $T$, there exists a nonnegative section $\zeta \in \Omega^{0}_{c}(T;\E|_{T})$ such that $K \subset \supp_{+}^{T} \xi$ and $d_{\nabla}\zeta \in \Lambda^{1}_{c,\Hol_{T}(\F,\E)}$.
\end{enumerate}
\end{thm}

We will prove Theorem \ref{thm:Hchar} after introducing integration along leaves. Let $(\E,\nabla)$ be a flat vector bundle over $M$. As in the case where $\E$ is trivial due to Haefliger~\cite[Theorem~3.1]{Haefliger1980}, the integration along the leaves of $\F$ 
\begin{equation}\label{eq:fint}
\fint_{\F} : \Omega^{\bullet}_{c}(M;\E) \longrightarrow \Omega^{\bullet-p}_{c}(\Tr \F ;\E)\;,
\end{equation}
is defined as follows. Take an open covering $\{U_{i}\}_{i \in I}$ of $M$ and a total transversal $T = \sqcup_{i \in I} T_{i}$ so that
\begin{enumerate}
\item $\{U_{i}\}_{i \in I}$ is locally finite,
\item $U_{i}$ is simply connected for $\forall i \in I$, 
\item the intersection of any plaque in $U_{i}$ and any plaque in $U_{j}$ is simply connected or empty for $\forall i$, $j \in I$,  
\item the intersection $P \cap T_{i}$ is a point for any plaque $P$ in $U_{i}$, $\forall i \in I$, and
\item the map $\pi_{i} : U_{i} \to T_{i}$ which maps each plaque $P$ in $U_{i}$ into $P \cap T_{i}$ is a disk bundle for $\forall i \in I$.
\end{enumerate}
Since $\E|_{U_{i}}$ and $\E|_{T_{i}}$ are trivial flat vector bundles for each $i$, we have
\begin{align*}
\Omega^{\bt}_{c}(U_{i};\E|_{U_{i}}) & \cong \Omega^{\bt}_{c}(U_{i})^{\oplus r}\;, & \Omega^{\bt}_{c}(T_{i};\E|_{T_{i}}) & \cong \Omega^{\bt}_{c}(T_{i})^{\oplus r}\;,
\end{align*}
where $r = \rank \E$. Then $\fint_{\pi_{i}} : \Omega^{\bt}_{c}(U_{i};\E|_{U_{i}})  \to \Omega^{\bt}_{c}(T_{i};\E|_{T_{i}})$ is defined by the $r$-times direct sum of the integration $\Omega^{\bt}_{c}(U_{i}) \to \Omega^{\bt}_{c}(T_{i})$ along $\pi_{i}$. Then~\eqref{eq:fint} is defined, for $\alpha \in \Omega^{k}_{c}(M;\E)$, by 
\[
\fint_{\F} \alpha = \sum_{i \in I} \fint_{\pi_{i}} \rho_{i} \alpha + \Lambda^{k - p}_{c, \Hol_{T}(\F,\E)}\;.
\] 
where $\{\rho_{i}\}_{i \in I}$ is a partition of unity subordinate to $\{U_{i}\}_{i \in I}$. As in the case where $\E$ is trivial~\cite[Theorem~3.1]{Haefliger1980}, we can check that $\fint_{\F}$ is a well-defined surjective continuous homomorphism which commutes with $d_{\nabla}$. By the argument in \cite[Theorem 3.2 in Section]{Haefliger1980}, we have the following.
\begin{lem}\label{lem:fibint}
The kernel of $\fint_{\F}$ is equal to the subspace of $\Omega^{\bullet}_{c}(M;\E)$ generated by $\F$-trivial forms and the differential of $\F$-trivial forms. For $i$, $j \in I$ with $U_{i} \cap U_{j} \neq \emptyset$, let $h_{ij}$ be the unique maximal element of $\Hol_{T}(\F,\E)$ such that $\Dom (h_{ij})_{T} \subset T_{i}$ and $\image (h_{ij})_{T} \subset T_{j}$, where $h_{T}$ is the map induced on $T$ by $h$. More generally, for $\omega \in \Omega^{\bt}(M;\E)$ which may not be compactly supported, if we have $\omega = \sum_{i}\omega_{i}$ and, for each $i$, we have $\fint_{\pi_{i}} \omega_{i} = \sum_{j}(h_{ij}^{*}\alpha_{ji} - \alpha_{ij})$ for some $\omega_{i} \in \Omega^{\bt}_{c}(U_{i};\E|_{U_{i}})$ and some $\alpha_{ji} \in \Omega^{\bt}_{c}(T_{j};\E|_{T_{j}})$, then $\omega =  \theta + d_{\nabla}\theta'$ for some $\F$-trivial forms $\theta$, $\theta'$.
\end{lem}

\begin{proof}[Outline of the proof of Theorem~\ref{thm:Hchar}]
Let us show (ii) by (i). Assume that $(M,\F)$ is strongly tense. By Proposition \ref{prop:tense}, there exists a relatively $\F$-closed form $\eta \in \Omega^{p}(M;\E)$ whose restriction to each leaf is nowhere vanishing. Take an open covering $\{U_{i}\}$ and a transversal $T$ to define the integration along leaves $\Omega^{\bullet}_{c}(M;\E) \to \Omega^{\bullet-p}_{c}(\Tr \F ;\E)$. By Lemmas \ref{lem:fibint}, it follows that $d_{\nabla} \fint_{\F} \eta = \fint_{\F}  d_{\nabla} \eta = 0$. Then, since $d_{\nabla}\sum_{i \in I} \fint_{\pi_{i}} \rho_{i} \alpha \in \Lambda^{1}_{c, \Hol_{T}(\F,\E)}$, we see that $\sum_{i \in I} \fint_{\pi_{i}} \rho_{i} \eta$ satisfies the condition of (ii). A twisted version of the proof of~\cite[Theorem in Section~3.2]{Haefliger1980} implies (iii) by (ii). The proof of (i) by (iii) is omitted because it is a special case of Theorem \ref{thm:domi2} below.
\end{proof}

We will show the following prescribed version of Theorem \ref{thm:domi}, which immediately implies Theorem \ref{thm:domi}.
\begin{thm}\label{thm:domi2}
Let $(M,\F)$ be a possibly open foliated manifold with a topologically trivial flat real line bundle $(\E,\nabla)$ with basic holonomy. If there exists a nonnegative $0$-cocycle $\xi$ in $\Omega^{0}_{lc}(\Tr \F;\E)$ such that $\supp_{+}^{\F}\xi = M$, then $(M,\F)$ is strongly tense with respect to $(\E,\nabla)$.
\end{thm}

The following is proven by an argument analogous to~\cite[Lemma~6.3]{AlvarezLopez1992}, which reduces the proof of Theorem \ref{thm:domi2} to the case where $\F$ is oriented.
\begin{prop}\label{prop:fincov}
Let $\pi : (M,\F) \to (M',\F')$ be a regular finite covering map of foliated manifolds with deck transformation group $\G$. Let $(\E',\nabla')$ be a flat real line bundle over $M'$. Let $(\E,\nabla) = (\pi^{*}\E', \pi^{*}\nabla')$. Then $(M,\F)$ is strongly tense with $(\E,\nabla)$ if and only if $(M',\F')$ is strongly tense with $(\E',\nabla')$. 
\end{prop}

\begin{proof}
The ``if'' part is trivial. We show the ``only if'' part. It suffices to prove the case where $\F$ is orientable. We assume that $\F$ is orientable and $(M,\F)$ is strongly tense with $(\E,\nabla)$. By Proposition~\ref{prop:tense}, we can assume that $\E'$ is topologically trivial. By Proposition \ref{prop:RS}, there exists a relatively $\F$-closed form $\eta \in \Omega^{p}(M;\E)$ whose restriction to any leaf is nowhere vanishing. Let us define $\varepsilon : \G \to \{\pm 1\}$ by $\varepsilon(\gamma)=1$ if $\gamma$ preserves the orientations of $\F$ and $\varepsilon (\gamma)=-1$ otherwise. Let $\eta^{\G} = \sum_{\gamma \in \G} \varepsilon (\gamma) \gamma^{*}(\eta)$. Then $\eta^{\G}$ is relatively $\F$-closed and the restriction to any leaf is nowhere vanishing. Moreover $\gamma^{*}\eta^{\G} = \varepsilon(\gamma) \eta^{\G}$ for any $\gamma \in \G$. Thus there exists a $\G$-invariant metric $g$ such that $\eta^{\G}$ is a local characteristic form at each point on $M$. This $g$ induces a strongly tense metric on $(M',\F')$, which concludes the proof by Proposition~\ref{prop:tense}.
\end{proof}

\begin{rem}
It is not clear that, for a finite covering map $(M,\F) \to (M',\F')$ of foliated manifolds, if $(M,\F)$ is tense, then $(M',\F')$ is tense.
\end{rem}

The basic idea of the following proof of Theorem \ref{thm:domi2} is to extend the support of a representative $\zeta$ of given $0$-cocycle $\xi$ in $\Omega^{0}_{lc}(\Tr \F;\E)$ so that we can construct a relatively $\F$-closed characteristic form from it. Haefliger's argument allows us to take $\zeta$ so that its support contains any compact subset in the transversal. But, in the case where the manifold is open, since characteristic forms are nowhere vanishing, roughly we need that the support of $\zeta$ contains a union of a countable number of compact disks. In this construction, we need to pay attention to certain convergence problem.

\begin{proof}[Proof of Theorem \ref{thm:domi2}]
By Proposition \ref{prop:fincov}, it suffices to show the strong tenseness of $\F$ in the case where $\F$ is oriented. Here we consider the case where $M/\overline{\F}$ is not compact. The proof of the other case is almost identical. By assumption, there exists a nonnegative $0$-cocycle $\xi^{0}$ in $\Omega^{0}_{lc}(\Tr \F;\E)$ such that $\supp_{+}^{\F}\xi^{0} = M$. Take a total transversal $T'$ such that the intersection of $T'$ and each leaf of $\F$ does not have accumulation points. Let $\zeta' \in \Omega^{0}_{lc}(T';\E|_{T'})$ be a nonnegative representative of $\xi^{0}$. Take a sequence of compact subsets $K_{0} \subset K_{1} \subset \cdots \subset K_{m} \subset \cdots$ of $T'$ such that $\cup_{j} K_{j} = \supp \zeta'$. Here, if $M/\overline{\F}$ is compact, then $\zeta'$ is compactly supported, and we can take $\supp \zeta'$ as $K_{0}$ and we do not need to take other $K_{j}$. Fix a Riemannian metric on $T$. For each $j$, there exists $\epsilon_{j} > 0$ such that any point on $K_{j}$ has a geodesically convex neighborhood of diameter $\epsilon_{j}$ in $T$. We cover $\overline{K_{j} \setminus K_{j-1}}$ by geodesically convex neighborhoods $T_{j}^{1}$, $\ldots$, $T_{j}^{N_{j}}$ of radius $\epsilon_{j}$ of a finite number of points $x^{1}_{j}$, $\ldots$, $x^{N_{j}}_{j}$ in $K_{j} \setminus K_{j-1}$. Reorder these geodesically convex neighborhoods, and denote them by $\{T_{i}\}_{i \leq 0}$. Let $T_{-} = \cup_{i \leq 0} T_{i}$. By a partition of unity on $T_{i}$ subordinate to $\{T_{i}\}_{i \leq 0}$, we obtain $\zeta^{0} \in \Omega^{0}_{lc}(T_{-};\E|_{T_{-}})$ which represents $\xi^{0}$. Take a geodesically convex subset $S_{i}$  for each $i$ so that $\overline{S_{i}} \subset T_{i}$ and $\zeta^{0}|_{S_{i}}$ is nowhere vanishing. Take open coverings $\{U_{i}\}_{i \in \ZZ}$, $\{V_{i}\}_{i \in \ZZ}$ of $M$ and transversals $T_{i}$, $S_{i}$ for $i \in \ZZ_{>0}$ as follows:
\begin{itemize}
\item Let $T = \sqcup_{i \in \ZZ} T_{i}$ and  $S = \sqcup_{i \in \ZZ} S_{i}$. Then the pair of $\{U_{i}\}_{i \in \ZZ}$ and $T$ (resp.\ $\{V_{i}\}_{i \in \ZZ}$ and $S$) satisfy the conditions (i), $\ldots$, (v) required to define the integration along leaves $\fint_{\F}$ in \eqref{eq:fint}.
\item $U_{i}$ and $V_{i}$ are relatively compact in $M$ and $\overline{V}_{i} \subset U_{i}$.
\item $S_{i} = V_{i} \cap T_{i}$ for $i \in \ZZ$.
\end{itemize}
For any $k \in \ZZ_{+}$, we can construct $\zeta^{k} \in \Omega^{0}_{lc}(T;\E|_{T})$ such that
\begin{enumerate}
\item  $\supp \zeta^{k} \subset \supp \zeta^{0} \cup T_{k}$, 
\item $S_{k} \subset \supp_{+}^{T} \zeta^{k}$ and 
\item $d_{\nabla} \zeta^{k} \in \Lambda^{1}_{lc, \Hol_{T}(\F,\E)}$
\end{enumerate}
as follows: Cover $S_{k}$ with a finite number of open subsets $\{W_{\lambda}\}_{\lambda \in A_{k}}$ of $T$ so that there exists $h_{\lambda} \in \Hol_{T}(\F,\E)$ such that $W_{\lambda} \subset \Dom (h_{\lambda})_{T}$ and $\image (h_{\lambda})_{T} \subset \supp_{+}^{T_{-}} \zeta^{0}$, where $(h_{\lambda})_{T}$ is the map induced on $T$ by $h_{\lambda}$. Take a smooth function $\rho_{\lambda}$ on $T_{-}$ such that $0 \leq \rho_{\lambda} < 1/\# A_{k}$ and $\rho_{\lambda}$ is nowhere vanishing on $(h_{\lambda})_{T}(W_{\lambda})$. Let 
\[
\zeta^{k} = \zeta^{0} + \sum_{\lambda \in A_{k}} h_{\lambda}^{*} (\rho_{\lambda} \zeta^{0}) - \sum_{\lambda  \in A_{k}} \rho_{\lambda} \zeta^{0}\;,
\]
which is nonnegative and $\zeta^{k} - \zeta^{0} \in \Lambda^{0}_{c, \Hol_{T}(\F,\E)}$. Then $\zeta^{k}$ satisfies the above conditions. For $(i,j) \in \ZZ^{2}$ with $U_{i} \cap U_{j} \neq \emptyset$, let $h_{ij}$ be the unique element of $\Hol_{T}(\F,\E)$ such that $\Dom (h_{ij})_{T} \subset T_{i}$ and $\image (h_{ij})_{T} \subset T_{j}$. We can decompose $\zeta^{k}$ as
\[
\zeta^{k} = \zeta^{0} - \sum_{i,j} (h_{ij}^{*} \alpha_{ji}^{k} - \alpha_{ji}^{k})
\]
for some $\alpha_{ji}^{k} \in \Omega_{c}^{0}(T_{j};\E|_{T_{j}})$ so that $\alpha_{ji}^{k} = 0$ except a finite number of $(i,j) \in \ZZ^{2}$. Let $a_{0}=1$ and, for $k > 0$, let
\[
a_{k} = \frac{1}{2^{k} \max_{(i,j)} \{ \|\alpha_{ji}^{k}\|_{k}, 1\}}\;,
\]
where $\|\cdot\|_{k}$ is a $C^{k}$-norm on $\Omega^{0}_{c}(T;\E|_{T})$. Since the limit $\lim_{N \to \infty} \sum_{0 \leq k \leq N} a_{k} \zeta^{k}$ exists in $\Omega^{0}(T;\E|_{T})$, let $\zeta = \lim_{N \to \infty} \sum_{0 \leq k \leq N} a_{k} \zeta^{k}$. Then $\zeta$ satisfies that $S \subset \supp_{+}^{T} \zeta$ and $d_{\nabla} \zeta = \sum_{i,j} h_{ij}^{*} \alpha_{ji} - \alpha_{ji}$ for some $\alpha_{ji} \in \Omega_{c}^{0}(T_{j};\E|_{T_{j}})$. For each $i$, we take a closed $p$-form $\chi_{i}$ on $U_{i}$ so that $V_{i} \subset \supp \chi_{i}$, $\fint_{\pi_{i}} \chi_{i} = 1$ and $\supp \chi_{i} \cap \pi^{-1}_{i}(x)$ is compact for any $x \in T_{i}$. Let $\eta_{i} =  \chi_{i} \otimes \pi_{i}^{*}(\zeta|_{T_{i}})$ and $\eta = \sum_{i} \eta_{i}$. Here we have 
\[
\fint_{\pi_{i}} d_{\nabla}\eta_{i} = d_{\nabla} \fint_{\pi_{i}} \eta_{i}  = d_{\nabla}\zeta|_{T_{i}} = \sum_{i,j} (h_{ij}^{*} \alpha_{ji} - \alpha_{ij})\;.
\]
Then, by applying Lemma \ref{lem:fibint} to $d_{\nabla}\eta = \sum_{i} d_{\nabla}\eta_{i}$, there exist $\F$-trivial forms $\theta$ and $\theta'$ such that $d_{\nabla}(\eta - \theta') = \theta$. It means that $\eta - \theta'$ is relatively $\F$-closed. Since $\eta - \theta'$ is nowhere vanishing on $M$ by construction, it follows that $(M,\F)$ is strongly tense with $(\E,\nabla)$ by Proposition \ref{prop:RS}.
\end{proof}

\section{Isomorphism between Haefliger cohomology and reduced Haefliger cohomology of Riemannian foliations}
\label{sec:4}

Let us recall some fundamental notions on pseudogroups. 

\begin{defn}
A pseudogroup $\H$ on a manifold $T$ is called {\em Riemannian} if there exists a Riemannian metric $g$ on $T$ such that $h^{*}g = g$ on $\Dom h$ for any $h\in \H$.
\end{defn}

The holonomy pseudogroup of a Riemannian foliation is Riemannian. The following notion is important in the context of Riemannian pseudogroups.

\begin{defn}\label{defn:comp}
A pseudogroup $\H$ on a manifold $T$ is called {\em complete} if, for any two points $x$, $y \in T$, there exists an open neighborhood $U_{x}$ (resp., $U_{y}$) of $x$ (resp., $y$) in $T$ such that, for every $h\in \H$ and $z\in U_{x} \cap \Dom h$ with $h(z) \in U_{y}$, there exists some $\tilde{h}\in \H$ so that $U_{x} \subset \Dom \tilde{h}$ and the germ of $\tilde{h}$ at $z$ is equal to the germ of $h$ at $z$. 
\end{defn}

A vector field $X$ on $T$ is called $\H$-{\em invariant} if $h_{*}X = X$ on $\image h$ for any $h\in \H$. 
\begin{defn}
A pseudogroup $\H$ on a manifold $T$ is called {\em locally parallelizable} if there exist $\H$-invariant vector fields $X_{1}$, $\ldots$, $X_{n}$ on $T$ such that $\{X_{1}, \ldots, X_{n}\}$ gives a trivialization of the tangent bundle of $T$. Such collection $\{X_{1}, \ldots, X_{n}\}$ of $\H$-invariant vector fields is called a {\em parallelism} of $\H$.
\end{defn}

We will prove the following result, a pseudogroup version of Theorem~\ref{thm:masa}, from which Theorem~\ref{thm:masa} immediately follows.
\begin{thm}\label{thm:masap}
Let $(\E,\nabla)$ be a flat vector bundle over a manifold $T$. Let $\H$ be a complete linear pseudogroup on $\E$ which induces a Riemannian pseudogroup $\H_{T}$ on $T$. Then the natural maps $H^{\bt}_{c}(\E/\H) \to \widehat{H}^{\bt}_{c}(\E/\H)$ and $H^{\bt}_{lc}(\E/\H) \to \widehat{H}^{\bt}_{lc}(\E/\H)$ are isomorphisms.
\end{thm}

We show the following as a byproduct, which immediately implies Theorem \ref{thm:findim}.
\begin{thm}\label{thm:findimp}
Under the hypothesis of Theorem \ref{thm:masap}, if moreover the space of orbit closures $T/\overline{\H_{T}}$ of $\H_{T}$ is compact, then we have $\dim H^{\bt}_{c}(\E/\H) < \infty$.
\end{thm}
Note that, if the space of orbit closures $T/\overline{\H_{T}}$ of $\H_{T}$ is compact, then we have that $H^{\bt}_{c}(\E/\H) \cong H^{\bt}_{lc}(\E/\H)$.

Let $n=\dim T$. First we consider the case where 
\begin{enumerate}
\item $\H_{T}$ is a locally parallelizable pseudogroup with parallelism $\{X_{1}, \ldots, X_{n}\}$ and
\item the space of orbit closures $T/\overline{\H_{T}}$ of $\H_{T}$ is compact.
\end{enumerate}
Since $T/\overline{\H_{T}}$ is compact, there exists a relatively compact set $U$ in $T$ whose $\H_{T}$-orbit is equal to $T$. Let $V$ be the real vector space generated by $X_{1}$, $\ldots$, $X_{n}$. We fix a compactly supported volume form $\mu$ on $V$. Then, making $\supp \mu$ smaller if necessary, we can assume that the time one map $\phi_{X} : U \to T$ of the flow generated by any $X\in \supp \mu$ is well-defined. Consider an operator
\begin{equation}\label{eq:s}
\begin{array}{cccc}
s : & \Omega_{c}^{\bt}(U;\E|_{U}) & \longrightarrow        & \Omega_{c}^{\bt}(T;\E) \\
& \alpha                & \longmapsto &  \int_{V} (\phi_{X}^{*}\alpha) \mu(X)\;.
\end{array}
\end{equation} 
Since $U$ is relatively compact, the following is well known.
\begin{lem}\label{lem:compact3}
The operator $s$ is compact.
\end{lem}
This operator $s$ is a transverse version of Sarkaria's smoothing operator~\cite{Sarkaria1978}. Masa~\cite{Masa1992} and Dom\'{i}nguez~\cite{Dominguez1998} used Sarkaria's smoothing operators, which are Hilbert-Schmidt integration operators on closed manifolds and hence compact~\cite[Eq.~13 of p.~694]{Sarkaria1978}. Note that even if $\H_{T}$ is the holonomy pseudogroup of a compact foliated manifold, $T$ may not be compact. In that case,  Hilbert-Schmidt integration operators on $T$ may not be compact in general.

We take $\{h_{i}\}_{i=1}^{\infty} \subset \H$ and open set $U_{i} \subset U \cap \Dom (h_{i})_{T}$ so that $\U=\{(h_{i})_{T}(U_{i})\}_{i=1}^{\infty}$ is a locally finite covering of $T$, where $(h_{i})_{T}$ is the map induced on $T$ by $h_{i}$. Let $\{\rho_{i}\}$ be a partition of unity on $T$ subordinate to $\U$. Consider a map
\[
\begin{array}{cccc}
\Phi_{\U} : & \Omega^{\bt}_{c}(T;\E) & \longrightarrow & \Omega^{\bt}_{c}(U;\E|_{U})\\
      &     \alpha   & \longmapsto & \sum_{i=1}^{\infty}h_{i}^{*}(\rho_{i}\alpha)\;.
\end{array}
\]
Let $\hat{s}$ be the composite
\begin{equation}\label{eq:So}
\xymatrix{ \Omega^{\bt}_{c}(T;\E) \ar[r]^<<<<<{\Phi_{\U}} & \Omega^{\bt}_{c}(U;\E|_{U}) \ar[r]^<<<<<{s} & \Omega^{\bt}_{c}(T;\E)\;. }
\end{equation}
Here $s$ is compact by Lemma~\ref{lem:compact3}. Since the composite of any continuous operator with a compact operator is compact, $\hat{s}$ is compact. We will apply the following fact.
\begin{thm}[see {\cite[Corollary~2 in Section~VIII]{RobertsonRobertson1980}}]\label{thm:des}
Let $B$ be a Hausdorff locally convex topological vector space and $\sigma : B \to B$ a compact operator. There exists $r\in \ZZ_{>0}$ such that
\begin{equation*}
B = \ker (1 - \sigma)^{r} \oplus \image (1 - \sigma)^{r}\;,
\end{equation*}
where $\ker (1 - \sigma)^{r}$ is closed and of finite dimension, $\image (1 - \sigma)^{r}$ is closed and
\begin{align*}
\ker (1 - \sigma)^{r} & = \ker (1 - \sigma)^{r+1}\;, \\
\image (1 - \sigma)^{r} & = \image (1 - \sigma)^{r+1}\;.
\end{align*}
\end{thm}
Then we get a positive integer $r$ and a decomposition
\begin{equation*}
\Omega^{\bt}_{c}(T;\E) = \K^{\bt} \oplus \I^{\bt}\;, 
\end{equation*}
where $\K^{\bt} = \ker (1 - \hat{s})^{r}$ and $\I^{\bt} = \image (1 - \hat{s})^{r}$ with the properties mentioned in Theorem~\ref{thm:des}. Since $\L^{\bt}_{c, \H}$ is locally convex and invariant under $\hat{s}$, Theorem~\ref{thm:des} also yields a similar decomposition
\begin{equation*}
\L^{\bt}_{c, \H} = (\L^{\bt}_{c, \H} \cap \K^{\bt}) \oplus (\L^{\bt}_{c, \H} \cap \I^{\bt} )\;.
\end{equation*}
Thus we get a decomposition
\begin{equation}\label{eq:dec0}
\Omega^{\bt}_{c}(\E/\H) = \Omega^{\bt}_{c}(T;\E)/\L^{\bt}_{c, \H} = \big( \K^{\bt} / (\L^{\bt}_{c, \H} \cap \K^{\bt}) \big) \oplus \big( \I^{\bt}/ (\L^{\bt}_{c, \H} \cap \I^{\bt}) \big)\;.
\end{equation}
It is well known (see, for example,~\cite[Theorem~9.1]{Treves1967}) that any finite dimensional Hausdorff topological vector space $B$ over $\RR$ is isomorphic to $\RR^{\dim B}$ with the product topology and hence it is closed. Then $\L^{\bt}_{c, \H} \cap \K^{\bt}$ is a closed subspace of $\K^{\bt}$, which implies that $\K^{\bt} / (\L^{\bt}_{c, \H} \cap \K^{\bt})$ is Hausdorff. Letting $\Xi^{\bt}=\K^{\bt} / (\L^{\bt}_{c, \H} \cap \K^{\bt})$ and $\Upsilon^{\bt}=\I^{\bt}/ (\L^{\bt}_{c, \H} \cap \I^{\bt})$, we get the following.
\begin{lem}\label{lem:decomposition}
There is a decomposition as a differential complex
\begin{equation}\label{eq:dec}
\Omega^{\bt}_{c}(\E/\H)  = \Xi^{\bt} \oplus \Upsilon^{\bt}
\end{equation}
such that 
\begin{enumerate}
\item $\Xi^{\bt}$ is Hausdorff of finite dimension,
\item $(1 - \hat{s})^{r}$ is the second projection $\Omega^{\bt}_{c}(\E/\H) \to \Upsilon^{\bt}$ and
\item $(1 - \hat{s})^{r+1}$ is a differential automorphism of $\Upsilon^{\bt}$.
\end{enumerate}
\end{lem}

Let $\oO^{\bt}$ be the closure of $0$ in $\Omega^{\bt}_{c}(\E/\H)$. Since $\Xi^{\bt}$ is Hausdorff by Lemma~\ref{lem:decomposition}-(i), $\oO^{\bt}$ is a subcomplex of $\Upsilon^{\bt}$. Thus, by taking the quotient of the both sides of~\eqref{eq:dec} by $\oO^{\bt}$, we get a decomposition of $\widehat{\Omega}^{\bt}_{c}(\E/\H)$. Here $\hat{s}$ induces a map $\widehat{\Omega}^{\bt}_{c}(\E/\H) \to \widehat{\Omega}^{\bt}_{c}(\E/\H)$, which is denoted by the same symbol $\hat{s}$. The following proposition is a direct consequence of Lemma~\ref{lem:decomposition}.
\begin{lem}\label{lem:decomposition1}
There is a decomposition of a differential complex
\begin{equation}\label{eq:dec2}
\widehat{\Omega}^{\bt}_{c}(\E/\H) = \Xi^{\bt} \oplus \Upsilon^{\bt}/\oO^{\bt}
\end{equation}
such that
\begin{enumerate}
\item $(1 - \hat{s})^{r}$ is the projection $\widehat{\Omega}^{\bt}_{c}(\E/\H) \to \Upsilon^{\bt}/\oO^{\bt}$ and
\item $(1 - \hat{s})^{r+1}$ is a differential automorphism of $\Upsilon^{\bt}/\oO^{\bt}$.
\end{enumerate}
\end{lem}

Let 
\begin{equation*}
\begin{array}{cccc}
h : & \Omega^{\bt}_{c}(T;\E) & \longrightarrow & \Omega^{\bt-1}_{c}(T;\E) \\
    & \alpha    & \longmapsto & -\int_{V} \Big( \int_{0}^{1} (\iota_{tX} \phi_{tX}^{*}\alpha) dt \Big) \mu(X)\;.
\end{array}
\end{equation*}
We show
\begin{lem}\label{lem:t}
For $\alpha\in \Omega^{\bt}_{c}(T;\E)$, we have
\begin{equation*}
\Big(\int_{V} \mu\Big)\alpha - s\alpha = (d_{\nabla}h + h d_{\nabla})\alpha\;. 
\end{equation*}
\end{lem}

\begin{proof}
For a vector field $X$ on $T$, the inner product and the Lie derivative with respect to $X$ on $\Omega^{\bt}_{c}(T;\E)$ is denoted by $\iota(X)$ and $\theta(X)$, respectively. Then
\begin{align*}
\Big(\int_{V} \mu \Big)\alpha - s\alpha & = - \int_{V} (\phi_{X}^{*}\alpha - \alpha) \mu(X) \\
& = -\int_{V} \left( \int_{0}^{1} \Big( \frac{d}{du}\Big|_{u = t} \phi_{uX}^{*}\alpha \Big)  dt \right) \mu(X) \\
& = -\int_{V} \left( \int_{0}^{1} (\theta(tX) \phi_{tX}^{*}\alpha) dt \right) \mu(X) \\
& = -\int_{V} \left( \int_{0}^{1} \big( (d_{\nabla} \iota(tX) + \iota(tX)d_{\nabla}) \phi_{tX}^{*}\alpha \big) dt \right) \mu(X) \\
& = (d_{\nabla}h + h d_{\nabla})\alpha\;.\qed 
\end{align*}
\renewcommand{\qed}{}
\end{proof}

Fix $\mu \in \Omega_{c}^{n}(V)$ so that $\int_{V} \mu = 1$. Since $\iota(X)$ commutes with the $\H$-action, we get a continuous map $h : \Omega^{\bt}_{c}(\E/\H) \to \Omega^{\bt}_{c}(\E/\H)$.
\begin{lem}\label{lem:Xi1} We have $H^{\bt}(\Upsilon^{\bt}) = 0$ and $ H^{\bt}(\Xi^{\bt}) \cong H^{\bt}_{c}(\E/\H)$.
\end{lem}
\begin{proof}
Note that $\hat{s}$ and $s$ induce the same map on $\Omega^{\bt}_{c}(\E/\H)$ by definition. Let $\alpha$ be a cocycle of $\Upsilon^{\bt}$. By Lemma~\ref{lem:decomposition}-(iii), there is a cocycle $\beta$ of $\Upsilon^{\bt}$ such that $(1 - s)^{r+1}\beta=\alpha$. By Lemma~\ref{lem:t}, we get 
\[ 
\alpha = (1 - s)^{r+1} \beta = (1 - s)^{r} (d_{\nabla}h +h d_{\nabla}) \beta = d_{\nabla} (1 - s)^{r} h \beta\;. 
\]
Here, $(1 - s)^{r} h \beta$ belongs to $\Upsilon^{\bt}$ by Lemma~\ref{lem:decomposition}-(ii). Thus $\alpha$ is a coboundary in $\Upsilon^{\bt}$, which implies $H^{\bt}(\Upsilon^{\bt}) = 0$. The second equation follows from $H^{\bt}(\Upsilon^{\bt}) = 0$ and Lemma~\ref{lem:decomposition}.
\end{proof}

Here $h$ induces a map $\widehat{\Omega}^{\bt}_{c}(\E/\H) \to \widehat{\Omega}^{\bt-1}_{c}(\E/\H)$ by its continuity. The following is proven in a way analogous to Lemma~\ref{lem:Xi1} by using Lemma~\ref{lem:decomposition1}. 
\begin{lem}\label{lem:Xi2} We have $H^{\bt}(\Upsilon^{\bt}/\oO^{\bt}) = 0$ and $ H^{\bt}(\Xi^{\bt}) \cong \widehat{H}_{c}^{\bt}(\E/\H)$.
\end{lem}

Thus, Theorem~\ref{thm:masap} for the case where $\H_{T}$ is locally parallelizable and $T/\overline{\H_{T}}$ is compact follows from Lemmas~\ref{lem:Xi1} and~\ref{lem:Xi2}.

Now, to consider the case where the base pseudogroup $\H_{T}$ is locally parallelizable, let us recall the structure theory of complete Riemannian pseudogroups due to Salem~\cite{Salem1988}, which is analogous to the Molino's structure theory of Riemannian foliations on compact manifolds. The following is an important class of locally parallelizable pseudogroups, which appears in Salem's theorem.
\begin{defn}
For a Lie group $G$, a pseudogroup $\H$ on $G$ is called of $G$-{\em Lie type} if $\H$ is generated by the left action of a subgroup $\G$ of $G$. If moreover $\G$ is dense in $G$, we say that $\H$ is of minimal $G$-Lie type.
\end{defn}
\begin{rem}
Note that this terminology is not common whereas the notion itself is common. Note also that the terminology {\em Lie pseudogroup} is commonly used for a different object.
\end{rem}
We need the following part of the analog of the Molino theory for complete Riemannian pseudogroups due to Salem. 
\begin{thm}[{\cite[Corollaires~1 and~2]{Salem1988}, see also~\cite[Appendix C by Salem]{Molino1988}}]\label{thm:Molino}
Let $\H$ be a complete Riemannian pseudogroup on a manifold $T$ whose orbit space $T/\H$ is connected. Let $T_{\#}$ be the orthonormal frame bundle over $T$. The pseudogroup $\H_{\#}$ on $T_{\#}$ naturally induced from $\H$ is complete and locally parallelizable. In general, for a complete locally parallelizable pseudogroup $\H_{\flat}$ on $T_{\flat}$ whose orbit space is connected, the following holds:
\begin{enumerate}
\item The space $T_{\flat}/\overline{\H_{\flat}}$ of the closures of the $\H_{\flat}$-orbits in $T_{\flat}$ is a smooth manifold and the projection $\pi : T_{\flat} \to T_{\flat}/\overline{\H_{\flat}}$ is a submersion. 
\item There exists a simply-connected Lie group $G$ such that, for each point of $x\in T_{\flat}/\overline{\H_{\flat}}$, there exists an open neighborhood $U$ of $x$ in $T_{\flat}/\overline{\H_{\flat}}$ such that $\H_{\flat}|_{\pi^{-1}(U)}$ is equivalent to the pseudogroup on $G \times U$ which is the product of a pseudogroup of minimal $G$-Lie type and the trivial pseudogroup on $U$.
\end{enumerate}
\end{thm}

For a complete linear pseudogroup on a flat vector bundle which induces a locally parallelizable pseudogroup on the base, we have the following direct consequence of Salem's theorem.
\begin{lem}\label{lem:Salem}
Let $(\E,\nabla)$ be a flat vector bundle over a manifold $T$. Let $\H$ be a complete linear pseudogroup on $\E$ which induces a locally parallelizable pseudogroup $\H_{T}$ on $T$. Assume that the orbit space $T/\H_{T}$ is connected. Let $\pi : T \to T/\overline{\H_{T}}$ be the canonical projection. Then, there exist a Lie group $G$, a flat vector bundle $\E_{G} \to G$, and a linear pseudogroup $\H_{G}$ on $\E_{G}$ which induces a pseudogroup of minimal $G$-Lie type on $G$ such that, for each point $z$ on $T$, there exists an open disk neighborhood $U$ of $\pi(z)$ in $T/\overline{\H_{T}}$ such that the pseudogroup $\H|_{\E|\pi^{-1}(U)}$ is equivalent to a pseudogroup $\H_{G} \times 1_{U}$ on $\E_{G} \boxtimes \RR_{U}$ which is the product of $\H_{G}$ and the trivial pseudogroup $1_{U}$ on $U$, where $\E_{G} \boxtimes \RR_{U}$ denotes the vector bundle $\E_{G} \times U \to G \times U$.
\end{lem}

To prove Theorem~\ref{thm:masap} for the case where $\H_{T}$ is locally parallelizable, we can assume that $T/\H_{T}$ is connected. Let $(\E,\nabla) \to T$ and $\H$ be as in the last lemma. There exist a Lie group $G$, a flat vector bundle $\E_{G} \to G$, and a linear pseudogroup $\H_{G}$ on $\E_{G}$ which satisfy the conditions of the last lemma. For each point $z$ on $T$, let $U$ be an open disk neighborhood of $\pi(z)$ in $T/\overline{\H_{T}}$ in the lemma. Let $\E_{U} = \E|_{\pi^{-1}(U)}$, $\H_{U} = \H|_{\E_{U}}$. Since $\H_{U}$ is equivalent to $\H_{G} \times 1_{U}$, we have
\begin{align*}
H^{\bt}_{c}(\E_{U}/\H_{U}) & \cong H^{\bt}_{c}\big((\E_{G} \boxtimes \RR_{U})/(\H_{G} \times 1_{U})\big)\;, \\ 
\widehat{H}^{\bt}_{c}(\E_{U}/\H_{U}) & \cong \widehat{H}^{\bt}_{c}\big((\E_{G} \boxtimes \RR_{U})/(\H_{G} \times 1_{U})\big)\;.
\end{align*}
Since the base pseudogroup of $\H_{G}$ is of minimal $G$-Lie type, we already proved Theorem~\ref{thm:masap} for $\H_{G}$. Then, since $\Om_{U}^{*} \otimes \Om_{U}$ is trivial, we have the following by Poincar\'{e} lemma:
\begin{multline*}
H^{\bt}_{c}\big((\E_{G} \boxtimes \RR_{U})/(\H_{G} \times 1_{U})\big) \cong H^{\bt}_{c}\big((\E_{G} \boxtimes \RR_{U}) \otimes \Om_{U}^{*} \otimes \Om_{U})/(\H_{G} \times 1_{U})\big) \cong \\
\cong H^{\bt - \dim U}_{c}\big((\E_{G} \otimes \Om_{U}^{*}) / \H_{G}\big) \cong \widehat{H}^{\bt - \dim U}_{c}\big((\E_{G} \otimes \Om_{U}^{*}) / \H_{G}\big) \\
\cong \widehat{H}^{\bt}_{c}\big((\E_{G} \boxtimes \RR_{U}) \otimes \Om_{U}^{*} \otimes \Om_{U})/(\H_{G} \times 1_{U})\big) \cong \widehat{H}^{\bt}_{c}\big((\E_{G} \boxtimes \RR_{U})/(\H_{G} \times 1_{U})\big)\;.
\end{multline*}
Thus we have $H^{\bt}_{c}(\E_{U}/\H_{U}) \cong \widehat{H}^{\bt}_{c}(\E_{U}/\H_{U})$. By taking a base of the topology of $T/\overline{\H_{T}}$ consisting of small disks, a standard argument with Mayer-Vietoris sequence implies that $H^{\bt}_{c}(\E/\H) \cong \widehat{H}^{\bt}_{c}(\E/\H)$ (see~\cite[Section 5.12]{GreubHalperinVanstone1972}). Since each element of $\Omega^{\bt}_{lc}\big((\E_{G} \boxtimes \RR_{U})/(\H_{G} \times 1_{U})\big)$ is represented by $\alpha \in \Omega^{\bt}(G \times U; \E_{G} \boxtimes \RR_{U})$ which is vertically compactly supported with respect to the second projection $G \times U \to U$, by Poincar\'{e} lemma, we have 
\begin{align*}
H^{\bt}_{lc}\big((\E_{G} \boxtimes \RR_{U})/(\H_{G} \times 1_{U})\big) & \cong H^{\bt}_{c}(\E_{G}/\H_{G})\;, \\
\widehat{H}^{\bt}_{lc}\big((\E_{G} \boxtimes \RR_{U})/(\H_{G} \times 1_{U})\big) & \cong \widehat{H}^{\bt}_{c}(\E_{G}/\H_{G})\;.
\end{align*}
Then the isomorphism $H^{\bt}_{lc}(\E/\H) \cong \widehat{H}^{\bt}_{lc}(\E/\H)$ is proven in a way similar to $H^{\bt}_{c}(\E/\H) \cong \widehat{H}^{\bt}_{c}(\E/\H)$. Thus Theorem~\ref{thm:masap} has been proven for the case where the base pseudogroup $\H_{T}$ is locally parallelizable.

Finally we will prove Theorem~\ref{thm:masap} in general, from which Theorem~\ref{thm:masa} follows. It will be done by a well known method based on a variant of the fact that the equivariant cohomology of a manifold with a free action of a connected Lie group is isomorphic to the cohomology of the quotient. See \cite[Theorem~3.7]{Dominguez1998} for similar arguments for Riemannian foliations.

Let $T_{\orr}$ be the orientation covering of $T$. Let $T_{\#}$ be a connected component of the orthonormal frame bundle of $T_{\orr}$. Let $\E_{\orr}$ (resp.\ $\E_{\#}$) be the flat vector bundles over $T_{\orr}$ (resp.\ $T_{\#}$) obtained by pulling back $\E$. Let $\H_{\orr}$ (resp.\ $\H_{\#}$) be the linear pseudogroup on $T_{\orr}$ (resp.\ $T_{\#}$) induced from $\H$. Let $\k=\so(n)$, where $n = \dim T$, and $A^{\bt} = \Omega^{\bt}_{c}(\E_{\#}/\H_{\#})$ and $\widehat{A}^{\bt} = \widehat{\Omega}^{\bt}_{c}(\E_{\#}/\H_{\#})$. Consider the Cartan complex 
\[
C_{\k}^{\bt}(A^{\bt}) := \big(({\textstyle \bigvee}\k^*) \otimes A^{\bt}\big)^{\k}\;,
\]
where $\bigvee \k^*$ denotes the symmetric algebra of $\k^{*}$. The differential $d_{\k}$ is defined by
\[
(d_{\k} \omega)(X):=d(\omega(X))-\iota(X)(\omega(X))
\]
for $X\in \k$, where we regard $\omega\in C_{\k}^{\bt}(A^{\bt})$ as a $\k$-equivariant polynomial map $\k\to A^{\bt}$. The $E^{1}$-terms of the spectral sequence of $(C_{\k}^{\bt}(A^{\bt}), d_{\k})$ with the filtration $F^{\ell} $ given by
\[
F^{\ell} = \bigoplus_{j \leq \ell} \big( ({\textstyle \bigvee{}^{j}} \k^{*}) \otimes A^{\bt} \big)^{\k}
\]
are $(\bigvee \k^{*})^{\k} \otimes H^{\bt}(A^{\bt})$, which are determined by $H^{\bt}(A^{\bt})$. We define $C_{\k}^{\bt}(\widehat{A}^{\bt})$ in a similar way. Since $\H_{\#}$ is complete and induces a locally parallelizable pseudogroup on $T_{\#}$ by Theorem~\ref{thm:Molino}, we already proved that the natural map $H^{\bt}(A^{\bt}) \to H^{\bt}(\widehat{A}^{\bt})$ is an isomorphism. Then it follows that $H^{\bt}(C_{\k}^{\bt}(A^{\bt})) \to H^{\bt}(C_{\k}^{\bt}(\widehat{A}^{\bt}))$ is an isomorphism, because it induces an isomorphism the $E_{1}$-terms. Here it is straightforward to see that $A^{\bt}$ admits the structures of a $\k$-dga with an algebraic connection. Since the $\k$-basic subalgebra $A^{\bt}_{\iota=0,\theta=0}$ of $A^{\bt}$ is isomorphic to $\Omega_{c}^{\bt}(\E_{\orr}/\H_{\orr})$, we have $H^{\bt}(C_{\k}^{\bt}(A^{\bt})) \cong H^{\bt}_{c}(\E_{\orr}/\H_{\orr})$  (see \cite[Theorem~IV in Section~8.17]{GreubHalperinVanstone1976}). We similarly have $H^{\bt}(C_{\k}^{\bt}(\widehat{A}^{\bt})) \cong \widehat{H}^{\bt}_{c}(\E_{\orr}/\H_{\orr})$. Then we have an isomorphism $H^{\bt}_{c}(\E_{\orr}/\H_{\orr}) \overset{\cong}{\to} \widehat{H}^{\bt}_{c}(\E_{\orr}/\H_{\orr})$. Clearly $\widehat{H}^{\bt}_{c}(\E/\H)$ and $H^{\bt}_{c}(\E/\H)$ are isomorphic to the $\ZZ/2\ZZ$-invariant subspace of $\widehat{H}^{\bt}_{c}(\E_{\orr}/\H_{\orr})$ and $H^{\bt}_{c}(\E_{\orr}/\H_{\orr})$, respectively. Hence the natural map induces $H^{\bt}_{c}(\E/\H) \cong \widehat{H}^{\bt}_{c}(\E/\H)$. The proof of the isomorphism $H^{\bt}_{lc}(\E/\H) \cong \widehat{H}^{\bt}_{lc}(\E/\H)$ is similar.

Here we outline the proof of Theorem \ref{thm:findimp}. If the base pseudogroup $\H_{T}$ of $\H$ is locally parallelizable and the space of orbit closures $T/\overline{\H_{T}}$ is compact, then Lemmas \ref{lem:decomposition}-(i) and \ref{lem:Xi1} imply that $\dim H^{\bt}_{c}(\E/\H) < \infty$. Let us consider the general case where $\H_{T}$ is Riemannian. Take $T_{\orr}$, $\E_{\orr}$, $\H_{\orr}$, $T_{\#}$, $\E_{\#}$ and $\H_{\#}$ as in the last paragraph. Since the base pseudogroup of $\H_{\#}$ is locally parallelizable by Salem's theorem (Theorem \ref{thm:Molino}) and $T_{\#}/\overline{\H_{\#}}$ is compact, $H^{\bt}_{c}(\E_{\#}/\H_{\#})$ is finite dimensional. Then, since the $E_{1}$-terms of the spectral sequences considered above is finite dimensional, it follows that the $E_{\infty}$-terms, which is isomorphic to $H^{\bt}_{c}(\E_{\orr}/\H_{\orr})$, are finite dimensional. Thus so is $H^{\bt}_{c}(\E/\H)$.

\section{Duality between reduced Haefliger cohomology and invariant cohomology of Riemannian foliations}\label{sec:6}

\subsection{Pairing with the invariant cohomology}

We recall the invariant cohomology of pseudogroups and the pairing with the Haefliger cohomology. Let $(\E,\nabla)$ be a flat vector bundle over an $n$-dimensional manifold $T$ with a linear pseudogroup $\H$. Let $\H_{T}$ be the base pseudogroup on $T$ induced by $\H$. Let $\Om_{T}$ be the orientation bundle of $T$. Note that $\H_{T}$ linearly acts on $\Om_{T}$. Recall that $\alpha \in \Omega^{\bt}(T;\E)$ is said to be {\em $\H$-invariant} if $h^{*} \alpha = \alpha$ on $\Dom h$ for any $h\in \H$. Let $\Omega^{\bt}_{\inv}(\H;\E)$ be the subcomplex of $\Omega^{\bt}(T;\E)$ consisting of $\H$-invariant forms. The cohomology of $\Omega^{\bt}_{\inv}(\H;\E)$ is called the invariant cohomology of $\H$ with values in $\E$ and denoted by $H^{\bt}_{\inv}(\H;\E)$. The wedge product
\begin{equation}\label{eq:pair1}
\Omega^{\bt}_{c}(T;\E) \times \Omega^{n-\bt}(T;\E^{*} \otimes \Om_{T}) \longrightarrow \Omega^{n}_{c}(T;\Om_{T})
\end{equation}
induces $\Omega^{\bt}_{c}(\E/\H) \times \Omega^{n-\bt}_{\inv}(\H;\E^{*} \otimes \Om_{T}) \longrightarrow \Omega^{n}_{c}(\Om_{T}/\H_{T})$ and hence 
\[
\widehat{\Omega}^{\bt}_{c}(\E/\H) \times \Omega^{n-\bt}_{\inv}(\H;\E^{*} \otimes \Om_{T}) \longrightarrow \widehat{\Omega}^{n}_{c}(\Om_{T}/\H_{T})\;,
\]
where $\widehat{\,\,}$ means the quotient by the closure of $0$. We will use the following observation.

\begin{prop}\label{prop:max}
Let $\H_{T}$ be a pseudogroup on an $n$-dimensional manifold $T$. 
Let $\Om_{T}$ be the orientation bundle of $T$. If the orbit space $T/\H_{T}$ is connected, then we have $H_{c}^{n}(\Om_{T}/\H_{T}) \cong \widehat{H}_{c}^{n}(\Om_{T}/\H_{T}) \cong \RR$.
\end{prop}

\begin{proof}
Let $L$ be the subspace of $H_{c}^{n}(T;\Om_{T})$ defined by
\begin{multline*}
L=\big\{ \, {\textstyle \sum}_{j=1}^{k} [\alpha_{j}] - h_{j}^{*}[\alpha_{j}] \, \big| \\
 h_{j} \in \H_{T}, \alpha_{j} \in \Omega^{\bt}_{c}(T;\Om_{T}), \pi^{-1}(\supp \alpha_{j}) \subset \Dom h_{j}, 1 \leq j \leq k \, \big\}\;,
\end{multline*}
where $\pi : \Om_{T} \to T$ is the bundle projection. By the last part of the cohomology exact sequence of 
\[
\xymatrix{0 \ar[r] & \L^{\bt}_{c, \H} \ar[r] & \Omega_{c}^{\bt}(T;\Om_{T}) \ar[r] & \Omega_{c}^{\bt}(\Om_{T}/\H_{T}) \ar[r] & 0\;,} 
\]
we get that $H_{c}^{n}(\Om_{T}/\H_{T}) \cong H_{c}^{n}(T;\Om_{T})/L$, which implies $H_{c}^{n}(\Om_{T}/\H_{T}) \cong \RR$ by assumption. Then $L$ is closed in $H_{c}^{n}(T;\Om_{T})$. So we have 
\[
\widehat{H}^{n}_{c}(\Om_{T}/\H_{T}) \cong H_{c}^{n}(T;\Om_{T})/\overline{L} \cong \RR\;.\qed
\]
\renewcommand{\qed}{}
\end{proof}

By Proposition~\ref{prop:max}, by decomposing $T/\H_{T}$ into the connected components, we have a natural map 
\begin{equation}\label{eq:pair2}
\Psi : \widehat{H}^{\bt}_{c}(\E/\H) \longrightarrow H_{\inv}^{n-\bt}(\H;\E^{*} \otimes \Om_{T})^{*}\;.
\end{equation}
Similarly we have
\begin{equation*}
\widehat{H}^{\bt}_{lc}(\E/\H) \longrightarrow H_{\inv, c}^{n-\bt}(\H;\E^{*} \otimes \Om_{T})^{*}\;,
\end{equation*}
where $H_{\inv, c}^{n-\bt}(\H;\E^{*} \otimes \Om_{T})$ is the cohomology of invariant forms such that the projection of the support to the space of orbit closures $T/\overline{\H_{T}}$ is compact, which is introduced in \cite{Sergiescu1985}.

Sections~\ref{sec:dp} and~\ref{sec:dR} will be devoted to proving the following result. 
\begin{thm}\label{thm:dualR}
Let $(\E,\nabla)$ be a flat vector bundle over an $n$-dimensional manifold $T$. Let $\H$ be a complete linear pseudogroup on $\E$ which induces a Riemannian pseudogroup $\H_{T}$ on $T$. Then we have 
\begin{align}
\widehat{H}^{\bt}_{c}(\E/\H) \cong & H_{\inv}^{n-\bt}(\H;\E^{*} \otimes \Om_{T})^{*}, \label{eq:dual2} \\
\widehat{H}^{\bt}_{lc}(\E/\H) \cong & H_{\inv, c}^{n-\bt}(\H;\E^{*} \otimes \Om_{T})^{*}. \notag
\end{align}
\end{thm}
The isomorphism \eqref{eq:dual2} is already known for the case of the trivial coefficient~\cite[p.~598]{AlvarezLopezMasa2008}. Theorem~\ref{thm:dfol} follows from this theorem, because the invariant cohomology of the holonomy pseudogroup of a foliated manifold $(M,\F)$ is naturally isomorphic to the basic cohomology of $(M,\F)$. Theorem~\ref{thm:dualR} will be proven by an argument similar to the proof of the twisted Poincar\'{e} duality of basic cohomology of Riemannian foliations due to Sergiescu~\cite[Th\'{e}or\`{e}me~I]{Sergiescu1985}.

\subsection{Duality for complete locally parallelizable pseudogroups}\label{sec:dp}

In this section, we show the following, which is the duality (Theorem~\ref{thm:dualR}) for complete linear pseudogroups on flat vector bundles which induce locally parallelizable pseudogroups on the base.
\begin{prop}\label{prop:dualp}
Let $(\E,\nabla)$ be a flat vector bundle over an $n$-dimensional manifold $T$. Let $\H$ be a complete linear pseudogroup on $\E$ which induces a locally parallelizable pseudogroup $\H_{T}$ on $T$. Then the pairing map $\Psi : \widehat{H}^{\bt}_{c}(\E/\H) \longrightarrow H_{\inv}^{n-\bt}(\H;\E^{*} \otimes \Om_{T})^{*}$ in~\eqref{eq:pair2} is an isomorphism.
\end{prop}

First we see the duality on the cochain level in the case where the base pseudogroup is of minimal $G$-Lie type.

\begin{lem}\label{lem:lied}
In addition to the hypothesis of the last proposition, assume further that, for a Lie group $G$, we have $T=G$ and $\H_{T}$ is of minimal $G$-Lie type. Then the map
\[
\Theta : \Omega^{\bt}_{c}(T;\E) \longrightarrow \Omega^{n-\bt}_{\inv}(\H;\E^{*} \otimes \Om_{T})^{*}
\]
defined by the pairing in~\eqref{eq:pair1} and the integration on $T$ induces an isomorphism $\widehat{\Omega}^{\bt}_{c}(\E/\H) \cong \big(\Omega^{n-\bt}_{\inv}(\H;\E^{*} \otimes \Om_{T})\big)^{*}$.
\end{lem}

\begin{proof}
First we show that $\Theta$ is surjective. Since the $\H_{T}$-orbits in $T$ are dense, $\H$-invariant forms valued in $\E^{*} \otimes \Om_{T}$ are determined by its value at the unit element $e \in G=T$. Thus we have $\dim \Omega^{n-\bt}_{\inv}(\H;\E^{*} \otimes \Om_{T}) < \infty$. Take a basis $\alpha_{1}$, $\ldots$, $\alpha_{N}$ of $\Omega^{n-\bt}_{\inv}(T;\E^{*} \otimes \Om_{T})$. Fix a density $\omega \in \Omega^{n}(T;\Om_{T})$ such that $\omega_{e} \neq 0$. Let $\{\rho_{k}\}$ be a sequence of compactly supported functions on $T$ such that $\{ f \mapsto \int_{T} f\rho_{k}\omega\} \subset \Omega^{0}(T)^{*}$ converges to the Dirac measure at $e$ as distributions and $\lim_{k \to \infty} \int_{T} \rho_{k} \omega > 0$. Take $\beta_{i} \in \Omega^{\bt}(T;\E)$ so that $(\beta_{i} \wedge \alpha_{j})_{e} = \delta_{ij} \omega_{e}$ for $i$, $j \in \{1, \ldots, N\}$. Then we have $\lim_{k \to \infty} \int_{T} \rho_{k} \beta_{i} \wedge \alpha_{j} = C \delta_{ij}$ for any $i$, $j \in \{1, \ldots, N\}$, where $C = \lim_{k \to \infty} \int_{T} \rho_{k} \omega$. Since $\rho_{k} \beta_{i} \in \Omega_{c}^{\bt}(T;\E)$, the subspace generated by $\Theta(\rho_{k} \beta_{1})$, $\ldots$, $\Theta(\rho_{k} \beta_{N})$ for $k \gg 0$ is equal to $\Omega^{n-\bt}_{\inv}(T;\E^{*} \otimes \Om_{T})^{*}$. Then $\Theta$ is surjective.

We will show that $\ker \Theta = \overline{\Lambda}^{\bt}_{c, \H}$. Recall that $\widehat{\Omega}_{c}^{\bt}(\E/\H) = \Omega^{\bt}_{c}(T;\E)/\overline{\Lambda}^{\bt}_{c, \H}$. By definition, it is easy to see that $\Lambda^{\bt}_{c, \H}$ is contained in $\ker \Theta$, and hence so is $\overline{\Lambda}^{\bt}_{c, \H}$. Take any $\alpha\in \Omega^{\bt}_{c}(T;\E)$ so that $\alpha \notin \overline{\Lambda}^{\bt}_{c, \H}$. By the Hahn-Banach theorem for $\Omega^{\bt}(T;\E)$ which is a Fr\'{e}chet space, there exists $v\in \Omega^{\bt}_{c}(T;\E)^{*}$ such that $v| \overline{\Lambda}^{\bt}_{c, \H} = 0$ and $v(\alpha) > 0$. Let $X_{1}$, $\ldots$, $X_{n}$ be an $\H$-invariant parallelism on $T$. Let $V$ be the real vector space generated by $X_{1}$, $\ldots$, $X_{n}$, and fix a volume form $\mu$ on $V$ of volume one such that the distribution $f \mapsto \int_{V}f\mu$ is sufficiently close to the Dirac measure at the origin of $V$. We will regularize this $v$ with a transverse version of the Sarkaria's operator $\hat{s} : \Omega^{\bt}_{c}(T;\E) \to \Omega^{\bt}_{c}(T;\E)$ (see \eqref{eq:So}). Recall that $\hat{s} = \Phi_{\mathcal{U}} \circ s$, where $\Phi_{\mathcal{U}}$ is an operator defined by a partition of unity defined by using a relatively compact set $U$ and $s$ is a diffusion operator defined by using $\mu$. Let $\hat{v} = v \circ \hat{s}$. Here, $\hat{v}$ satisfies
\begin{enumerate}
\item $\hat{v}| \overline{\Lambda}^{\bt}_{c, \H} = 0$,
\item $\hat{v}(\alpha) > 0$ and
\item $\exists v^{*} \in \Omega^{n-\bt}(T;\E^{*} \otimes \Om_{T})$ such that  $\hat{v} (\beta) = \int_{T} \beta \wedge v^{*}$ for $\forall \beta \in  \Omega^{\bt}_{c}(T;\E)$. 
\end{enumerate}
Indeed, since the parallelism $X_{1}$, $\ldots$, $X_{n}$ is $\H$-invariant, $s$ commutes with the $\H$-action. Thus we have $\hat{s} (\overline{\Lambda}^{\bt}_{c, \H}) \subset \overline{\Lambda}^{\bt}_{c, \H}$, which implies that $\hat{v}| \overline{\Lambda}^{\bt}_{c, \H} = 0$. Since the distribution associated with $\mu$ is close to the Dirac measure at the origin of $V$ as a distribution, we have $v(\hat{s} \alpha)$ is close to $v(\Phi_{\mathcal{U}} \alpha)$, which implies that $v(\Phi_{\mathcal{U}} \alpha) > 0$. Since $v(\Phi_{\mathcal{U}} \alpha) = v (\alpha)$ by $\Phi_{\mathcal{U}} \alpha - \alpha \in \Lambda^{\bt}_{c, \H}$ and $v| \overline{\Lambda}^{\bt}_{c, \H} = 0$, it follows that $\hat{v}(\alpha) > 0$. The property (iii) follows from the fact that $s$ is a diffusion operator (see \cite[Section 15]{deRham1984}). The $\H$-invariance of $\hat{v}$ implies that $v^{*}$ is $\H$-invariant. Thus, by $\hat{v}(\alpha) > 0$, it implies that $\alpha \notin \ker \Theta$. Thus we have that $\ker \Theta = \overline{\Lambda}^{\bt}_{c, \H}$, which implies that $\Theta$ induces an isomorphism $\widehat{\Omega}^{\bt}_{c}(\E/\H) \cong \Omega^{n-\bt}_{\inv}(\H;\E^{*} \otimes \Om_{T})^{*}$.
\end{proof}

\begin{cor}\label{cor:dualpc}
Under the hypothesis of Lemma \ref{lem:lied}, the pairing map $\widehat{H}^{\bt}_{c}(\E/\H) \longrightarrow H_{\inv}^{n-\bt}(\H;\E^{*} \otimes \Om_{T})^{*}$ in~\eqref{eq:pair2} is an isomorphism. 
\end{cor}

\begin{proof}
Since the cohomologies of these complexes are of finite dimension by Theorem~\ref{thm:findimp}, the duality $\widehat{H}^{\bt}_{c}(\E/\H) \cong H^{\bt}_{\inv}(\H;\E)$ follows from the universal coefficient theorem and the last lemma. 
\end{proof}

Now let us prove Proposition \ref{prop:dualp}. We can assume that $T/\H_{T}$ is connected. Thus there exists a Lie group $G$, a flat vector bundle $\E_{G} \to G$, and a linear pseudogroup $\H_{G}$ on $\E_{G}$ which satisfy the following conditions of Lemma \ref{lem:Salem}: for each point $z$ on $T$, let $U$ be an open disk neighborhood of $\pi(z)$ in $T/\overline{\H_{T}}$ in the lemma. Let $\E_{U} = \E|_{\pi^{-1}(U)}$ and $\H_{U} = \H|_{\E_{U}}$. The pseudogroup $\H_{U}$ is equivalent to a pseudogroup $\H_{G} \times 1_{U}$ on $\E_{G} \boxtimes \RR_{U}$. Thus the Haefliger cohomology and invariant cohomology of $\H_{U}$ are isomorphic to those of $\H_{G} \times 1_{U}$. Since the base pseudogroup of $\H_{G}$ is of minimal $G$-Lie type, we can apply Corollary \ref{cor:dualpc}. Then, since $\Om_{U}^{*} \otimes \Om_{U}$ is trivial and $\Om_{T} = \Om_{U} \otimes \Om_{G}$, by Poincar\'{e} lemmas, we have
\begin{multline}\label{eq:P}
\widehat{H}^{\bt}_{c}\big((\E_{G} \boxtimes \RR_{U})/(\H_{G} \times 1_{U})\big) \cong 
\widehat{H}^{\bt}_{c}\big(((\E_{G} \boxtimes \RR_{U}) \otimes  \Om_{U}^{*} \otimes \Om_{U})/(\H_{G} \times 1_{U})\big) \\
\cong \widehat{H}^{\bt - \dim U}_{c}\big((\E_{G} \otimes  \Om_{U}^{*})/\H_{G} \big) \cong 
\widehat{H}^{\dim G + \dim U - \bt}_{\inv}\big(\H_{G} ; \E_{G}^{*}  \otimes \Om_{U} \otimes \Om_{G} \big) \\
\cong H^{n - \bt}_{\inv}(\H_{G}; \E_{G}^{*} \otimes \Om_{T})^{*} \cong H^{n - \bt}_{\inv}\big(\H_{G} \times 1_{U}; (\E_{G} \boxtimes \RR_{U})^{*} \otimes \Om_{T} \big)^{*}.
\end{multline}
Thus we have $\widehat{H}^{\bt}_{c}(\E_{U}/\H_{U}) \cong H^{n - \bt}_{\inv}(\H_{U}; \E_{U}^{*} \otimes \Om_{T})^{*}$. Now the isomorphism $\widehat{H}^{\bt}_{c}(\E/\H) \cong H^{n - \bt}_{\inv}(\H; \E^{*}\otimes \Om_{T})^{*}$ follows from a standard argument with Mayer-Vietoris sequence with a base of the topology of $T/\overline{\H}$ consisting of open disks (see~\cite[Section 5.12]{GreubHalperinVanstone1972}). The proof of $\widehat{H}^{\bt}_{lc}(\E/\H) \cong H^{n - \bt}_{\inv, c}(\H; \E^{*}\otimes \Om_{T})^{*}$ is similar with the following computation, which corresponds to \eqref{eq:P}: 
\begin{multline*}
\widehat{H}^{\bt}_{lc}\big((\E_{G} \boxtimes \RR_{U})/(\H_{G} \times 1_{U})\big) \cong \widehat{H}^{\bt}_{c}(\E_{G}/\H_{G}) \cong H^{\dim G - \bt}_{\inv}(\H_{G} ; \E_{G}^{*} \otimes \Om_{G})^{*} \\ \cong H^{\dim G + \dim U - \bt}_{\inv, c}(\H_{G} \times 1_{U}; (\E_{G} \boxtimes \RR_{U})^{*} \otimes \Om_{G} \otimes \Om_{U})^{*} \cong H^{n - \bt}_{\inv, c}(\H_{U};\E_{U}^{*} \otimes \Om_{T})^{*}\,. 
\end{multline*}

\subsection{Duality for complete Riemannian pseudogroups}\label{sec:dR}

Let $\H$ be a linear complete pseudogroup on a flat vector bundle $(\E,\nabla) \to T$ which induces a Riemannian pseudogroup on $T$. Here we will deduce the duality (Theorem~\ref{thm:dualR}) for $\H$ from Proposition~\ref{prop:dualp} by comparing two spectral sequences in a way similar to~\cite[D\'{e}montration du th\'{e}or\`{e}me~I]{Sergiescu1985} and~\cite[Sections~4.7 and~4.8]{RoyoPrietoSaralegiWolak2009}. Let $T_{\orr}$ be the orientation covering of $T$. Let $\pi : T_{\#} \to T_{\orr}$ be a connected component of the orthonormal frame bundle of $T_{\orr}$. Let $\E_{\orr}$ (resp.\ $\E_{\#}$) be the flat vector bundles over $T_{\orr}$ (resp.\ $T_{\#}$) obtained by pulling back $\E$. Let $\H_{\orr}$ (resp.\ $\H_{\#}$) be a linear pseudogroup on $\E_{\orr}$ (resp.\ $\E_{\#}$) induced by $\H$. Let $n=\dim T$ and $n_{\#} = \dim T_{\#}$. We similarly define $\E^{*}_{\#}$, $\E^{*}_{\orr}$, $\Om_{T_{\#}}$ and $\Om_{T_{\orr}}$. First, we show that  the pairing map
\begin{equation}
\Psi : \widehat{H}^{\bt}_{c}(\E_{\orr}/\H_{\orr}) \longrightarrow H^{n-\bt}_{\inv}(\H_{\orr};\E^{*}_{\orr} \otimes \Om_{T_{\orr}})^{*}\; 
\end{equation}
for $\H_{\orr}$ (see \eqref{eq:pair2}) is an isomorphism. Consider a complex 
\begin{align*}
C^{\bt}_{\#} = \Omega^{n_{\#} - \bt}_{\inv}(\H_{\#};\E^{*}_{\#}\otimes \Om_{T_{\#}})^{*}\;, 
\end{align*}
whose differential is dual to $d_{\nabla}$. Let $K = \SO(n)$ and $\k = \so(n)$. Since the $K$-action on $\E_{\#}$ commutes with $\H_{\#}$, the complexes $\widehat{\Omega}^{\bt}_{c}(\E_{\#}/\H_{\#})$ and $C^{\bt}_{\#}$ admit $K$-actions. These $K$-action induce the structure of $\k$-dgas on these complexes, which admits an algebraic connection. Here, recall that the spectral sequence associated with a $\k$-dga $A^{\bt}$ is given by the $\k$-invariant subalgebra $A_{\theta=0}^{\bt} = \{ \, \alpha \in A^{\bt} \mid \theta(X)(\alpha) = 0, \forall X \in \k \, \}$ with the filtration given by
\[
F^{s}A_{\theta=0}^{s+t} = \{ \, \alpha \in A_{\theta=0}^{s+t} \mid \iota(X_{0}) \circ \cdots \circ \iota(X_{t}) \alpha = 0, X_{i} \in \k \, \}\;.
\]
The $E_{2}^{s,t}$-term is isomorphic to $\widehat{H}^{s}(A^{\bt}_{\iota=0, \theta=0}) \otimes H^{t}(\k)$,where $A^{\bt}_{\iota=0, \theta=0}$ is the $\k$-basic subalgebra given by $A_{\iota=0, \theta=0} = \{ \, \alpha \in A^{\bt} \mid \iota(X)(\alpha) = 0, \theta(X)(\alpha) = 0, \forall X \in \k \, \}$ (see, for example, \cite[Corollary III in Section 9.5]{GreubHalperinVanstone1976}). Let $E_{r}^{s,t}$ be the spectral sequence associated with $\widehat{\Omega}^{\bt}_{c}(\E_{\#}/\H_{\#})$, which converges to $\widehat{H}^{\bt}_{c}(\E_{\#}/\H_{\#})$. Since $\widehat{\Omega}^{\bt}_{c}(\E_{\#}/\H_{\#})$ admits an algebraic connection and its $\k$-basic subalgebra is isomorphic to $\widehat{\Omega}^{\bt}_{c}(\E_{\orr}/\H_{\orr})$, we have
\begin{equation}\label{eq:e21} 
E_{2}^{s,t} \cong \widehat{H}^{s}_{c}(\E_{\orr}/\H_{\orr}) \otimes H^{t}(\k)\;.
\end{equation}
Similarly, consider the spectral sequence $\overline{E}_{r}^{s,t}$ associated with $C^{\bt}_{\#}$, which converges to $H^{n_{\#} - \bt}_{\inv}(\H_{\#};\E^{*}_{\#} \otimes \Om_{T_{\#}})^{*}$. The $\k$-basic subalgebra is isomorphic to $\Omega^{n-\bt}_{\inv}(\H_{\orr};\E^{*}_{\orr}\otimes \Om_{T_{\orr}})^{*}$ (see \cite[Lemme 2.8]{Sergiescu1985}) and we have 
\begin{equation}\label{eq:e22}
\overline{E}_{2}^{s,t} \cong H^{s}_{\inv}(\H_{\orr};\E^{*}_{\orr} \otimes \Om_{T_{\orr}})^{*} \otimes H^{t}(\k)\;. 
\end{equation}
Here $\Psi$ induces an isomorphism between the $E_{2}^{0,t}$-terms by $\widehat{H}^{0}_{c}(\E_{\#}/\H_{\#}) \cong H^{n_{\#}}_{\inv}(\H_{\#}; \E^{*}_{\#}\otimes \Om_{T_{\#}})^{*}$ and the following lemma. 

\begin{lem}\label{lem:comm}
For each $s$, we have a commutative diagram
\[
\xymatrix{ 
\widehat{\Omega}^{s}_{c}(\E_{\#}/\H_{\#})^{\k} \ar[rr]^<<<<<<<<<<{\Psi} && \big(\Omega^{n_{\#} - s}_{\inv}(\H_{\#};\E^{*}_{\#}\otimes \Om_{T_{\#}})^{*}\big)^{\k} \\
\widehat{\Omega}^{s}_{c}(\E_{\orr}/\H_{\orr})  \ar[rr]_<<<<<<<<<<<<{\Psi} \ar[u]^{\pi^{*}} && \Omega^{n - s}_{\inv}(\H_{\orr};\E^{*}_{\orr}\otimes \Om_{T_{\orr}})^{*} \ar[u]_{\textstyle (\fint_{\pi})^{*}}\;, 
}
\]
whose vertical maps are isomorphisms in the case where $s=0$. 
\end{lem}

\begin{proof}
To prove the commutativity, it suffices to show that $\int_{T_{\#}} \pi^{*}\alpha \wedge \beta = \int_{T} \alpha \wedge \big(\fint_{\pi} \beta\big)$ for $\alpha \in \Omega^{s}_{c}(T_{\orr};\E_{\orr})$ and $\beta \in \Omega^{n_{\#} -s}(T_{\#};\E_{\#})$. This equality follows from an equality $\fint_{\pi} (\pi^{*}\alpha \wedge \beta) = \alpha \wedge \fint_{\pi} \beta$, which is a well known property of the integration along fibers (see, for example, \cite[Proposition IX in Section 7.13]{GreubHalperinVanstone1972}). Let us consider the case where $s=0$. It is easy to see that $\pi^{*} : \Omega^{0}_{c}(T_{\orr};\E_{\orr}) \to \Omega^{0}_{c}(T_{\#};\E_{\#})^{\k}$ is an isomorphism and so is $\pi^{*} : \widehat{\Omega}^{0}_{c}(\E_{\orr}/\H_{\orr}) \to \widehat{\Omega}^{0}_{c}(\E_{\#}/\H_{\#})^{\k}$. Fix $\gamma \in \wedge^{\dim \k}\k^{*} \otimes \Om_{K}$ such that $\int_{K}\gamma = 1$. Let $\gamma_{\#}$ be the differential form on $T_{\#}$ valued in $\Om_{K}$ whose kernel is equal to the horizontal plane field of Levi-Civita connection, and whose restriction to each fiber of $\pi$ is equal to $\gamma$. We can see that any $\beta$ in $\Omega^{n_{\#}}(T_{\#};\E^{*}_{\#}\otimes \Om_{T_{\#}})^{\k}$ is of the form $\pi^{*}\beta' \wedge \gamma_{\#}$ for some $\beta' \in \Omega^{n}(T_{\orr};\E^{*}_{\orr}\otimes \Om_{T_{\orr}})$. By $\fint_{\pi} (\pi^{*}\beta' \wedge \gamma_{\#}) = \beta'$, we have that $\fint_{\pi}$ induces $\Omega^{n_{\#}}_{\inv}(\H_{\#}; \E^{*}_{\#}\otimes \Om_{T_{\#}})^{\k} \cong \Omega^{n}_{\inv}(\H_{\orr};\E^{*}_{\orr}\otimes \Om_{T_{\orr}})$. It follows that $(\fint_{\pi})^{*} : \Omega^{n}_{\inv}(\H_{\orr};\E^{*}_{\orr}\otimes \Om_{T_{\orr}})^{*} \to \big(\Omega^{n_{\#}}_{\inv}(\H_{\#}; \E^{*}_{\#}\otimes \Om_{T_{\#}})^{\k}\big)^{*}$ is an isomorphism. Here, by using the Hahn-Banach theorem for $\Omega^{n_{\#}}_{\inv}(\H_{\#}; \E^{*}_{\#}\otimes \Om_{T_{\#}})$, which is Frech\'{e}t, and the averaging operator on $K$, it is easy to see that the restriction map induces an isomorphism
\[
\big(\Omega^{n_{\#}}_{\inv}(\H_{\#}; \E^{*}_{\#}\otimes \Om_{T_{\#}})^{*}\big)^{\k} 
\cong
\big(\Omega^{n_{\#}}_{\inv}(\H_{\#}; \E^{*}_{\#}\otimes \Om_{T_{\#}})^{\k}\big)^{*}\;.
\]
Thus we have  $(\fint_{\pi})^{*}$ is an isomorphism.
\end{proof}

Since $\H_{\#}$ is complete and the base pseudogroup induced by $\H_{\#}$ is locally parallelizable, Proposition~\ref{prop:dualp} implies that $\Psi$ induces an isomorphism $\widehat{H}^{\bt}_{c}(\E_{\#}/\H_{\#}) \cong H^{n_{\#}-\bt}_{\inv}(\H_{\#}; \E^{*}_{\#} \otimes \Om_{T_{\#}})^{*}$ between the $E_{\infty}$-terms of the two spectral sequences. Then, by Zeeman's comparison theorem \cite{Zeeman1957}, $\Psi$ induces an isomorphism $E_{2}^{s,0} \cong \overline{E}_{2}^{s,0}$ for any $s$. Then, by \eqref{eq:e21} and~\eqref{eq:e22}, it follows that $\Phi$ is an isomorphism $\widehat{H}^{\bt}_{c}(\E_{\orr}/\H_{\orr}) \longrightarrow H^{n-\bt}_{\inv}(\H_{\orr};\E^{*}_{\orr} \otimes \Om_{T_{\orr}})^{*}$.

Since $\widehat{H}^{\bt}_{c}(\E/\H)$ and $H^{n-\bt}_{\inv}(\H;\E^{*} \otimes \Om_{T})^{*}$ are the $\ZZ/2\ZZ$-invariant part of the both sides of $\widehat{H}^{\bt}_{c}(\E_{\orr}/\H_{\orr}) \longrightarrow H^{n-\bt}_{\inv}(\H_{\orr};\E^{*}_{\orr} \otimes \Om_{T_{\orr}})^{*}$, respectively, it follows that $\widehat{H}^{\bt}_{c}(\E/\H) \cong H^{n-\bt}_{\inv}(\H;\E^{*} \otimes \Om_{T})^{*}$. The proof of $\widehat{H}^{\bt}_{lc}(\E/\H) \cong H^{n-\bt}_{\inv, c}(\H;\E^{*} \otimes \Om_{T})^{*}$ is similar.


\section{Applications}

We will show some direct consequences obtained by combing Theorems~\ref{thm:masap} and~\ref{thm:dualR} with known results on invariant cohomology.

For a complete linear pseudogroup $\H$ on a flat vector bundle $(\E,\nabla)$ over a manifold $T$ which induces a Riemannian pseudogroup on $T$, provided that the orbit space of the base pseudogroup is connected, the pseudogroup version of the argument due to El Kacimi Alaoui-Sergiescu-Hector~\cite{ElKacimiAlaouiSergiescuHector1985} implies that $H^{0}_{\inv}(\H;\E) \cong \RR$ or $0$. Then, Theorems~\ref{thm:masap} and~\ref{thm:dualR} imply the following.

\begin{cor}
Let $(\E,\nabla)$ be a flat vector bundle over a manifold $T$ with a complete linear pseudogroup $\H$ on $\E$ which induces a Riemannian pseudogroup $\H_{T}$ on $T$. If $T/\H_{T}$ is connected and $\rank \E=1$, then $H^{0}_{c}(\E/\H)$ is isomorphic to $\RR$ or $0$.
\end{cor}

The following corollary is the pseudogroup version of Corollary~\ref{cor:inv}, from which Corollary~\ref{cor:inv} follows.

\begin{cor}\label{cor:inv2}
Let $(\E,\nabla)$ be a flat vector bundle over an $n$-dimensional manifold $T$. Let $\H$ be a linear pseudogroup on $\E$ which induces a complete Riemannian pseudogroup $\H_{T}$ on $T$, and $\P$ the Sergiescu's orientation sheaf of $\H_{T}$. Assume that the orbit space $T/\H_{T}$ is connected. Then we have the following.
\begin{enumerate}
\item If $\H$ is complete, then we have $H^{\bt}_{c}(\E/\H) \cong H^{n-\bt}_{lc}((\E^{*} \otimes \P)/\H)^{*}$.
\item $H^{0}_{lc}(\P/\H_{T}) \cong \RR$.
\item $H^{0}_{c}(\P/\H_{T}) \cong \RR$ if and only if the space of orbit closures $T/\overline{\H_{T}}$ is compact.
\end{enumerate}
\end{cor}
Corollary~\ref{cor:inv2} follows from Theorems~\ref{thm:masap} and \ref{thm:dualR}, and the following result on the invariant cohomology of complete Riemannian pseudogroups.
\begin{thm}[(ii) is due to {\cite[Proposition~3.2.9.1]{Haefliger1985}}]\label{thm:fin}
Under the same assumption with Corollary~\ref{cor:inv2}, the following holds: 
\begin{enumerate}
\item $H^{\bt}_{\inv}(\H;\E) \cong H^{n-\bt}_{\inv, c}(\H;\E^{*} \otimes \P)^{*}$.
\item $H^{n}_{\inv, c}(\H; \P) \cong \RR$.
\item $H^{n}_{\inv}(\H; \P) \cong \RR$ if and only if the space of orbit closures $T/\overline{\H_{T}}$ is compact.
\end{enumerate}
\end{thm}
Theorem~\ref{thm:fin}-(i) is proven by the pseudogroup version of the argument of the proof of Sergiescu~\cite[Th\'{e}or\`{e}me I]{Sergiescu1985}. If $T/\overline{\H_{T}}$ is compact, then we have $H^{n}_{\inv}(\H; \P) \cong H^{n}_{\inv, c}(\H; \P) \cong \RR$. To show the converse, note that $H^{n}_{\inv}(\H,\P) \cong H^{0}_{\inv, c}(\H;\P)^{*}$ by (ii). Here $H^{0}_{\inv, c}(\H;\P)$ is the space of $\H$-invariant parallel sections of $\P$ which is compactly supported. Thus, if $T/\overline{\H_{T}}$ is not compact, then we have $H^{0}_{\inv, c}(\H;\P) \cong 0$, which implies the converse.

The invariant cohomology of complete Riemannian pseudogroups with trivial coefficient is a topological invariant by \cite[Corollary 27.3 and Theorem 29.1]{AlvarezLopezMasa2008}. Topological invariance of the invariant cohomology with values in flat vector bundles is similarly proven by using their results. Combining it with Theorems \ref{thm:masap} and \ref{thm:dualR}, we have the following, which says that $H^{\bt}_{c}(\E/\H)$ and $H^{\bt}_{lc}(\E/\H)$ are topological invariants for complete Riemannian pseudogroups.
\begin{cor}
Let $\pi_{i} : (\E_{i},\nabla_{i}) \to T_{i}$ be a flat vector bundle with a complete linear pseudogroup $\H_{i}$ $(i=1,2)$. We have $H^{\bt}_{c}(\E_{1}/\H_{1}) \cong H^{\bt}_{c}(\E_{2}/\H_{2})$ and $H^{\bt}_{lc}(\E_{1}/\H_{1}) \cong H^{\bt}_{lc}(\E_{2}/\H_{2})$ if there exist an open covering $\{V^{j}\}$ of $T_{1}$, an open covering $\{W^{j}\}$ of $T_{2}$ and a $C^{0}$ bundle homeomorphism $\Phi^{j} : \E_{1}|_{V^{j}} \to \E_{2}|_{W^{j}}$ such that  
\begin{enumerate}
\item $\Phi^{j}$ maps local parallel sections of $(\E_{1},\nabla_{1})|_{V^{j}}$ to those of $(\E_{2},\nabla_{2})|_{W^{j}}$  and
\item $\{\Phi^{j}\}$ generates an equivalence between $\H_{1}$ and $\H_{2}$ in the sense of Haefliger.
\end{enumerate}
\end{cor}

Finally we prove the following result, which implies Corollaries \ref{cor:st} and \ref{cor:masa}.

\begin{thm}
Any complete Riemannian foliation is strongly tense with respect to the Sergiescu's orientation sheaf.
\end{thm}

This theorem follows from Theorem \ref{thm:domi2}, Proposition~\ref{prop:fincov} and the following lemma. 

\begin{lem}
Let $\H$ be a complete Riemannian pseudogroup on an orientable manifold $T$. Let $\P$ be the Sergiescu's orientation sheaf of $\H$ over $T$. Then there exists a nonnegative section $\zeta \in \Omega^{0}_{lc}(T;\P)$ such that $\supp_{+}^{T}\zeta$ intersects all $\H$-orbits.
\end{lem}

\begin{proof}
Note that $\H$ has a natural linear action on $\P$. First consider the case where $\H$ is of minimal $G$-Lie type for a Lie group $G$. In this case, $\P$ is the flat real line bundle over $G$ given as $\det \mg^{-}$, where $\mg^{-}$ is the Lie algebra of right invariant vector fields on $G$ by definition of Molino's commuting sheaf. By the description of $\H$-invariant cochains in Eq.\ \eqref{eq:co0} and by Lemma~\ref{lem:lied}, the space of $\H$-invariant $0$-cochains $\widehat{\Omega}^{0}_{c}(\P/\H) \cong \Omega^{0}_{\inv}(\H;\P)^{*} \cong (\Omega^{0}(G;\P)^{G})^{*} \cong \big(\det (\mg^{-})^{G}\big)^{*}$ is at most of dimension $1$. Then, since $\widehat{H}^{0}_{c}(\P/\H) \cong \widehat{H}^{0}_{lc}(\P/\H) \cong \RR$ by Corollary \ref{cor:inv2}-(ii), the projection of any element $\Omega^{0}_{c}(G;\P)$ to $\widehat{\Omega}^{0}_{c}(\P/\H)$ is closed. Take a nonnegative section $\zeta$ of $\P$, which is not zero at a point. Then $d_{\nabla}\zeta \in \overline{\Lambda}^{1}_{c, \H}$. Consider an exact sequence 
\[
\xymatrix{ 0 \ar[r] & \overline{\Lambda}^{\bt}_{c, \H}/\Lambda^{\bt}_{c, \H} \ar[r] & \Omega_{c}^{\bt}(\P/\H) \ar[r]^{\tau} & \widehat{\Omega}_{c}^{\bt}(\P/\H) \ar[r] & 0\;.}
\]
Since $\tau$ induces an isomorphism on the cohomology by Theorem \ref{thm:masap}, it follows that $\overline{\Lambda}^{\bt}_{c, \H}/\Lambda^{\bt}_{c, \H}$ is acyclic. By $d_{\nabla}(d_{\nabla}\zeta)=0$, it follows that the projection $d_{\nabla}\hat{\zeta}$ of $d_{\nabla}\zeta$ to $\overline{\Lambda}^{1}_{c, \H}/\Lambda^{1}_{c, \H}$ is closed. Then there exists $\hat{\alpha} \in \overline{\Lambda}^{0}_{c, \H}/\Lambda^{0}_{c, \H}$ such that $d_{\nabla}\hat{\alpha} = d_{\nabla}\hat{\zeta}$. Then, taking a lift $\alpha$ of $\hat{\alpha}$, we have $d_{\nabla}(\zeta - \alpha) \in \Lambda^{0}_{c, \H}$. 

Here we will modify $\zeta$ and $\alpha$ to obtain a nonnegative element in $\Omega^{0}_{c}(T;\P)$ whose differential belongs to $\Lambda^{1}_{c, \H}$. Since $\Lambda^{0}_{c, \H}$ is dense in $\overline{\Lambda}^{0}_{c, \H}$, there exists a sequence $\{\alpha_{i}\}$ in $\Lambda^{0}_{c, \H}$ which converges to $\alpha$. By the definition of the topology, there exists a relatively compact subset $K$ of $G$ such that $\supp \alpha_{i} \subset K$ for large enough $i$ and $\lim_{i \to \infty}\|\alpha - \alpha_{i}\|_{k} = 0$ for any $k$, where $\|\cdot\|_{k}$ is a $C^{k}$-norm on $\Omega_{c}^{\bt}(K;\P|_{K})$. Here, by using an argument in \cite[Proof of Theorem 4.1]{Haefliger1980}, we will see that there exists $\theta \in \Lambda^{0}_{c, \H}$ such that $\zeta + \theta$ is positive on $K$. Take $h_{1}$, $\ldots$, $h_{\ell} \in \H$ so that 
\begin{enumerate}
\item $K \subset \cup_{j} \Dom h_{j}$,
\item $\zeta|_{h_{j}(U_{j})}$ is positive and
\item there exists a compact subset $A_{j} \subset \Dom h_{j}$ $(j=1,\ldots,\ell)$ such that $K \subset \cup_{j} A_{j}$.
\end{enumerate}
Take a partition of unity $\{\rho_{j}\}_{j=1}^{\ell} \cup \{\tau\}$ subordinate to $\{h_{j}(U_{j})\} \cup \{T\setminus \cup_{j}h_{j}(A_{j})\}$ such that $\rho_{j}$ is positive on $h_{j}(A_{j})$ $(j=1,\ldots,\ell)$. Let 
\[
\theta = \sum_{j=1}^{\ell} h_{j}^{*} (\rho_{j}\zeta) - \sum_{j=1}^{\ell} \rho_{j}\zeta\;.
\] 
Then $\theta \in \Lambda^{0}_{c, \H}$ and $\zeta + \theta$ is positive on $K$. Since $\lim_{i \to \infty}\|\alpha - \alpha_{i}\|_{0} = 0$, it follows that $\zeta + \theta - (\alpha - \alpha_{i})$ is nonnegative for $i \gg 0$. We also have $d_{\nabla}(\zeta + \theta - (\alpha - \alpha_{i})) = d_{\nabla}(\zeta - \alpha) + d_{\nabla}\theta - d_{\nabla}\alpha_{i} \in \Lambda^{1}_{c, \H}$. Then the lemma is proved for the case where $\H$ is of minimal $G$-Lie type.

We can assume that $T/\H$ is connected to prove the result. Consider the case where $\H$ is complete and locally parallelizable. Let $\pi : T \to T/\overline{\H}$ be the canonical projection. Since $T/\H$ is connected, by Lemma \ref{lem:Salem}, there exists a Lie group $G$, a flat vector bundle $\P_{G} \to G$ and a linear pseudogroup $\H_{G}$ on $\P_{G}$ which satisfy the following condition: for each point $z$ on $T$, there exists an open disk neighborhood $U$ of $\pi(z)$ in $T/\overline{\H}$ such that the pseudogroup $\H|_{\P|\pi^{-1}(U)}$ is equivalent to a pseudogroup on $\P_{G} \boxtimes \RR_{U}$ which is the product of $\H_{G}$ and the trivial pseudogroup on $U$. Here, since the linear $\H_{G}$-action on $\P_{G}$ is equivalent to the linear $(\H|_{\pi^{-1}(z)})$-action on $\P|_{\pi^{-1}(z)}$, and $\P|_{\pi^{-1}(z)}$ is the Sergiescu's orientation sheaf of $\H|_{\pi^{-1}(z)}$, it follows that $\P_{G}$ is the Sergiescu's orientation sheaf of $\H_{G}$, which is of minimal $G$-Lie type.

Thus we have a locally finite open covering $\{U_{s}\}$ of $T/\overline{\H}$ such that the pseudogroup $\H|_{\P|\pi^{-1}(U_{s})}$ is equivalent to a pseudogroup $\H_{G} \times 1_{U_{s}}$ on $\P_{G} \boxtimes \RR_{U_{s}}$. Let $\P_{U_{s}} = \P|_{\pi^{-1}(U_{s})}$ and $\H_{U_{s}} = \H|_{\P_{U_{s}}}$. Let $\Phi = \{\Phi_{s}^{j}\}$ be the equivalence from $\H_{U_{s}}$ to $\H_{s} \times 1_{U_{s}}$, which consists of isomorphisms of flat vector bundles. Let $\{\upsilon_{s}^{j}\}$ be the partition of unity on $G \times U_{s}$ subordinate to $\{\image (\Phi_{s}^{j})_{T}\}$, where $(\Phi_{s}^{j})_{T}$ is the map induced on $T$ by $\Phi_{s}^{j}$. Consider
\[
\begin{array}{cccc}
\Phi_{s}^{*} & : \Omega^{0}_{lc}(G \times U_{s} ; \P_{G} \boxtimes \RR_{U_{s}}) & \longrightarrow & \Omega^{0}_{lc}(\pi^{-1}(U_{s});\P_{U_{s}})  \\
& \sigma & \longmapsto & (\Phi_{s}^{j})^{*} \upsilon_{s}^{j}\sigma\,.
\end{array}
\]
Take a partition of unity $\{\rho_{s}\}$ on $T/\overline{\H}$ subordinate to $\{U_{s}\}$. Since $\H_{s}$ is of minimal $G$-Lie type, there exists a nonzero nonnegative element $\zeta_{s}$ in $\Omega^{0}_{c}(\P_{s}/\H_{s})$ such that $d_{\nabla}\zeta_{s} \in \Lambda^{1}_{c, \H_{s}}$. Let $\zeta = \sum_{s} \Phi_{s}^{*} \big( (\pi^{*}\rho_{s}) (p_{1}^{*}\zeta_{s}) \big)$, where $p_{1} : G \times U_{s} \to G$ is the first projection. Then $\zeta$ is a nonzero nonnegative section in $\Omega^{0}_{lc}(T;\P)$ such that $d_{\nabla}\zeta \in \Lambda^{1}_{lc, \H}$ and $\supp_{+}^{T}\zeta$ intersects all orbits of $\H$. 

Finally consider the general case where $\H$ is a complete Riemannian pseudogroup. Let $T_{\#}\to T$ be a connected component of the orthonormal frame bundle over $T$. Let $\H_{\#}$ be the pseudogroup on $T_{\#}$ induced from $\H$. Let $\P_{\#}$ be the Sergiescu's orientation sheaf of $\H_{\#}$. Since $\H_{\#}$ is locally parallelizable by Salem's theorem (Theorem \ref{thm:Molino}), there exists a nonnegative section $\zeta_{\#}$ in $\Omega^{0}_{lc}(T_{\#};\P_{\#})$ such that $d_{\nabla}\zeta_{\#} \in \Lambda^{1}_{lc, \H_{\#}}$ and $\supp_{+}^{T_{\#}}\zeta_{\#}$ intersects any orbit of $\H_{\#}$. Here, by definition of Molino's central sheaf, it is easy to see that $\P$ is obtained as the quotient of $\P_{\#}$ by the natural $\SO(n)$-action (see \cite[Section 5.3]{Molino1988}). By integrating $\zeta_{\#}$ along the fibers of $T_{\#}\to T$, we get $\zeta$ which satisfies the conditions.
\end{proof}

\section{Examples of open foliated manifolds}\label{sec:9}

We illustrate difference of tautness or tenseness between closed and open foliated manifolds with some examples. 

In Theorem \ref{cor:masa}, we generalized one direction of Masa's characterization of tautness of Riemannian foliations~\cite[Minimality theorem]{Masa1992} to complete Riemannian foliations. But the other direction can fail for simple examples of Riemannian foliations on open manifolds.

\begin{prop}\label{prop:lieex}
Let $\G$ be a finitely generated subgroup of a Lie group $G$. Then the pseudogroup generated by the left $\G$-action on $G$ is equivalent to a holonomy pseudogroup of a taut foliation on an open manifold.
\end{prop}

\begin{proof}
Take a surjective homomorphism $F_{m} \to \G$, where $F_{m}$ is a free group of rank $m$. Take a genus $2m$ closed surface $\Sigma$ and the suspension foliation for $\pi_{1}\Sigma \to F_{m} \to \Gamma$. Since any suspension foliation is taut, the proof is concluded.
\end{proof}

If $G$ is not unimodular, then we have $H_{c}^{0}(\Tr \F) = H_{lc}^{0}(\Tr \F) = 0$ for any minimal $G$-Lie foliation $\F$ by Corollary~\ref{cor:Lie}. By Proposition~\ref{prop:lieex}, we get the following.
\begin{ex}\label{ex:con}
There are taut Riemannian foliations $\F$ on open manifolds $M$ such that $H_{c}^{0}(\Tr \F)  = H_{lc}^{0}(\Tr \F) = 0$. Equivalently, there are taut Riemannian foliations $\F$ on open manifolds $M$ whose Sergiescu's orientation sheaf is nontrivial.
\end{ex}

As already mentioned in the introduction, a tense metric on a closed manifold with a Riemannian foliation is always strongly tense by~\cite[Eq.~4.4]{KamberTondeur1983}, while the example of Cairns-Escobales~\cite[Example~2.4]{CairnsEscobales1997} shows that it is not true in general for the open case. There is a further difference in this direction. The cohomology class of the mean curvature form of a tense metric on a Riemannian foliation on a compact manifold. More generally, the \'{A}lvarez class, the cohomology class of the basic component of the mean curvature form $\kappa$, is well-defined for a Riemannian foliation on a compact manifold by~\cite[Theorem~5.2]{AlvarezLopez1992}. However the following example shows that the cohomology classes of the closed basic mean curvature form depend on metrics for Riemannian foliations on open manifolds.

\begin{prop}
Consider a foliation $\F$ on $\RR \times S^{1}$ whose leaves are the fibers of the second projection $\RR \times S^{1} \to S^{1}$. For any $t\in \RR^{\times}$, there exists a strongly tense bundle-like metric on $(M,\F)$ whose mean curvature form $\kappa$ satisfies $\int_{S^{1}}\kappa = t$.
\end{prop}

\begin{proof}
For $t\in \RR^{\times}$, let $\chi_{t}$ be a flat $\RR$-connection form whose holonomy homomorphism is $h_{t}: \pi_{1}S^{1} \to \Aut(\RR)$ determined by $h_{t}(\gamma) = t$, where $\gamma$ is a generator of $\pi_{1}S^{1}$. Then, by Rummler's formula~\eqref{eq:R}, $\chi_{t}$ is the characteristic form of a strongly tense metric on $(M,\F)$ such that the cohomology class of the mean curvature form $\kappa$ is given by $\int_{\gamma} \kappa = \log |t|$.
\end{proof}

Note that Royo Priento-Saralegi-Wolak \cite{RoyoPrietoSaralegiWolak2018} proved that the \'{A}lvarez class is well-defined for singular Riemannian foliations.

In~\cite{Nozawa2012}, it was proven that tautness of one dimensional Riemannian foliations on closed manifolds is invariant under deformation. This is not true in general for complete Riemannian manifolds on open manifolds as the following example shows. 

\begin{ex}
Let $A \in \SL(2;\ZZ)$ be a hyperbolic matrix and $T^{3}_{A}$ be the mapping torus. Let $\mathcal{G}$ be the non-taut Riemannian flow on $T^{3}_{A}$ considered by Carri\`{e}re~\cite{Carriere1984}. Let $\F_{0}$ be the Riemannian flow on $T^{3}_{A} \times \RR$ such that $\F|_{T^{3}_{A} \times \{s\}} = \mathcal{G}$ for any $s$. Let $\F_{1}$ be the Riemannian flow on $T^{3}_{A} \times \RR$ defined by $T^{3}_{A} \times \RR = \bigsqcup_{x \in T^{3}_{A}} \{x\}  \times \RR$. Let $X$ be a nonsingular vector field tangent to $\mathcal{G}$. For $0 \leq t \leq 1$, consider a vector field $X_{t}$ on $T^{3}_{A} \times \RR$ given by 
\[
X_{t}(x,s)= X(x) \cos (t\pi/2) + \frac{d}{ds} \sin(t\pi/2),
\]
which defines a Riemannian flow $\F_{t}$ on $T^{3}_{A} \times \RR$. Here $\F_{t}$ is taut for $0 < \forall t \leq 1$. It is easy to see that $\F_{t}$ is an $\RR$-bundle and hence it is taut for for $0 < t \leq 1$, while $\F_{0}$ is not taut.
\end{ex}

\end{document}